%% file: main.tex
\documentclass[journal]{IEEEtran}

\usepackage{amsmath,amssymb,amsfonts,amsthm}
\usepackage{mathrsfs}
\usepackage{lipsum}
\usepackage{algorithmic}
\usepackage{graphicx}
\usepackage{textcomp}
\usepackage[ruled,vlined]{algorithm2e}
\usepackage{verbatim}
\usepackage{tabularx}
\usepackage{array}
\usepackage{enumitem}
\usepackage{cite}
\usepackage{url}
\usepackage{color}

\newtheorem{theorem}{Theorem}
\newtheorem{lemma}{Lemma}
\newtheorem{definition}{Definition}
\newtheorem{corollary}{Corollary}
\newtheorem*{remark}{Remark}
\newtheorem{assumption}{Assumption}
\newtheorem{proposition}{Proposition}

\newcommand{\bigroman}[1]{\uppercase\expandafter{\romannumeral#1}}

\begin{document}

\title{On Inhomogeneous Infinite Products of Stochastic Matrices and Applications}

\author{Zhaoyue Xia, \IEEEmembership{Student Member, IEEE},
        Jun Du, \IEEEmembership{Senior Member, IEEE}, Chunxiao Jiang, \IEEEmembership{Fellow, IEEE}, \\ H. Vincent Poor, \IEEEmembership{Fellow, IEEE}, Zhu Han, \IEEEmembership{Fellow, IEEE}, Yong Ren, \IEEEmembership{Senior Member, IEEE}
         % <-this % stops a space
        %Lajos~Hanzo,~\IEEEmembership{Life~Fellow,~IEEE}
\thanks{Z. Xia, J. Du, and Y. Ren are with the Department of Electronic Engineering, Tsinghua University, Beijing, 100084, China, (e-mail: xiazy19@mails.tsinghua.edu.cn; \{jundu,reny\}@tsinghua.edu.cn).} % <-this % stops a space
\thanks{C. Jiang is with Tsinghua Space Center, Tsinghua University, Beijing, 100084, China. (e-mail: jchx@tsinghua.edu.cn).}
\thanks{H. V. Poor is with the Department of Electrical and Computer Engineering, Princeton University, Princeton, NJ 08544 USA, (e-mail: poor@princeton.edu).} 
\thanks{Z. Han is with the Department of Electrical and Computer Engineering at the University of Houston, Houston, TX 77004 USA, and also with the Department of Computer Science and Engineering, Kyung Hee University, Seoul, South Korea, 446-701, (e-mail: hanzhu22@gmail.com).}} % <-this % stops a space}

\maketitle
\begin{abstract}
  With the growth of magnitude of multi-agent networks, distributed optimization holds considerable significance within complex systems. Convergence, a pivotal goal in this domain, is contingent upon the analysis of infinite products of stochastic matrices (IPSMs). In this work, convergence properties of inhomogeneous IPSMs are investigated. The convergence rate of inhomogeneous IPSMs towards an absolute probability sequence $\pi$ is derived. We also show that the convergence rate is nearly exponential, which coincides with existing results on ergodic chains. The methodology employed relies on delineating the interrelations among Sarymsakov matrices, scrambling matrices, and positive-column matrices. Based on the theoretical results on inhomogeneous IPSMs, we propose a decentralized projected subgradient method for time-varying multi-agent systems with graph-related stretches in (sub)gradient descent directions. The convergence of the proposed method is established for convex objective functions, and extended to non-convex objectives that satisfy Polyak-Lojasiewicz conditions. To corroborate the theoretical findings, we conduct numerical simulations, aligning the outcomes with the established theoretical framework. 
\end{abstract}


\input{sections/introduction.tex}
\input{sections/model.tex}
\input{sections/application.tex}
\input{sections/simulation.tex}
\input{sections/conclusion.tex}

% \appendix
% \input{sections/Appendix.tex}

\bibliographystyle{IEEEtran}
\bibliography{IEEEabrv,ref}

\end{document}

%% file: sections/introduction.tex
\section{Introduction}
As multi-agent systems become increasingly prevalent in various domains ranging from social networks to industrial production lines, the need for efficient and adaptable decentralized optimization techniques has increased in importance \cite{2023Event-Triggered,2023FullyDist,2023Positive,2023TrackingControlMultiagent,2023DistributedProximalGradient}. Among various fields, multi-agent systems with time-varying topologies have garnered significant attention in the area of decentralized optimization. Traditional centralized optimization techniques often struggle to adapt to dynamic environments due to their reliance on centralized decision-making and static network structures. In contrast, decentralized approaches leverage the collective intelligence and autonomy of individual agents within a network, enabling robustness, scalability, and adaptability to evolving graph structures. 

The fusion of decentralized multi-agent systems and time-varying topologies presents both theoretical and practical challenges in the field of optimization. Firstly, a notable theoretical gap persists in understanding the exceptional efficacy exhibited by first-order sampling, or by an abuse of terminology, distributed stochastic (sub)gradient descent (DSG), within time-varying graphs. In addition, how to develop an effective topology-dependent decentralized algorithm that ensures convergence and superior performance remains a challenge. 

This paper addresses these challenges by revisiting the theory related to inhomogeneous infinite products of stochastic matrices (IPSMs). Additionally, topology-dependent variations of gradient descent directions are introduced within the decentralized algorithm to promote both convergence and satisfactory performance. Based on ergodic theory, our method investigates the links between optimization objectives and mechanics of gradient descent. 

A distinctive contrast between centralized and decentralized optimization methodologies prominently emerges in the critical phase known as consensus. In the centralized setting, a common center is responsible for disseminating current optimization parameters to all agents. Conversely, within the decentralized framework, each agent is tasked with autonomously achieving consensus throughout the progression of the algorithm. Indeed, the essential condition for achieving consensus with respect to DSG methods fundamentally hinges on the specific attributes of the time-varying topology, which is in turn a subject related to ergodic theory. Since each weighted directed graph can be associated with a (row-)stochastic matrix by normalizing, understanding the asymptotic properties of inhomogeneous IPSMs stands as a significant area of investigation. Although characteristics of homogenous products of stochastic matrices are straightforward to obtain with the help of the Perron-Frobenius theorem and spectrum theory, analyzing their inhomogeneous counterpart revokes the more complex mathematical tool known as coefficients of ergodicity developed in \cite{1979CoefErgodicity}. Extensive work on inhomogenous products of stochastic matrices has focused on the exponential convergence of scrambling matrices \cite{1963ProdSIAMatrices} and positive-column matrices \cite{1977ExpConvProdStochasticMat} in the context of weak ergodicity, i.e., in the form of $\lim_{k\rightarrow \infty} A(k)\cdots A(2)A(1)$. However, a deeper understanding of the truncated product represented as $A(k)\cdots A(s+1)A(s)$ for any $s < k$ is in need to study the convergence property of corresponding decentralized algorithms, which calls for explicit estimates of $\Phi_A(s,k) = A(k)\cdots A(s+1)A(s)$.

The classical form of a standard projected DSG (PSG) method generates iterates $(\mathbf{x}_{i,k})_{i\in \mathcal{V}, k\in \mathbb{N}}$ through $
  \mathbf{x}_{i,k+1} = \mathcal{P}_{\mathcal{X}} \left[ \sum_{j=1}^n A_{ij}(k) \mathbf{x}_{j,k} - \alpha_k \mathbf{v}(\mathbf{x}_{i,k}) \right],
$
where $\mathcal{X}$ is the projection space and $\mathbf{v}(\mathbf{x}_{i,k})$ is a descent direction at $\mathbf{x}_{i,k}$ for the local objective function $f_i$. As a typical decentralized algorithm, such standard PSG method does not introduce any topological variation in the descent direction. As a consequence, an intuitive motivation arises suggesting that leveraging graph information within desent directions could potentially accelerate convergence rates. Within the realm of machine learning, introducing topology-variant stretches to the gradient direction substantiates the rationale behind incorporating a topology-related regularization term into the optimization objective. This perspective, from a geometrical standpoint, effectively transforms the topology-invariant solution space into a dynamic, time-varying space. 

Based on the analysis and motivation above, we summarize our contributions of this work as follows:
\begin{itemize}
  \item We present results on the convergence of inhomogeneous IPSMs under weaker conditions than \cite{1963ProdSIAMatrices,2005AgreementConsensusProblems}. Specifically, the stochastic matrices neither are selected from any compact set $P$ nor belong to the $S(\delta)$ category\footnote{A matrix $A\in S(\delta)$ means that all nonzero elements are larger than or equal to $\delta$ for some constant $\delta>0$.}. 
  \item The convergence rate for the truncated product $\Phi(s,k)=A(k)\cdots A(s+1)A(s)$ towards an absolute probability sequence $\pi(s)$ is derived. We also show that the convergence rate is nearly exponential, which paves the way for stronger conclusions in the vanishing step size regime.
  \item We investigate the asymptotic behavior of the absolute probability sequence when the stochastic matrix sequence tends to the identity matrix. Theoretical analysis shows that the absolute probability sequence $\pi(k)$ tends to $\frac{1}{n}\mathbf{1}^T$.
  \item A decentralized projected subgradient method is developed which adds graph-related stretches to the desent directions. The convergence to the global minima of non-smooth convex objective functions is established based on the results on inhomogeneous IPSMs. Moreover, we generalize the convergence results to the case where objective functions are non-convex and satisfy the Polyak-Lojasiewicz (PL) conditions. 
  \item Numerical simulations are conducted to validate the theoretical findings established in this study. Moreover, the outcomes of these simulations affirm that the incorporation of graph-related stretches into the descent directions leads to notably enhanced performance compared to employing the standard decentralized PSG method. 
\end{itemize}

The rest of this article is organized as follows: we first review related works on inhomogeneous IPSMs and DSG over time-varying graphs in Section \ref{sec:realted-works}, then present our theoretical results in Section \ref{sec:theoretical-results}. Based on these results, we propose a decentralized subgradient method and establish its convergence in Section \ref{sec:app}. Numerical simulation results are presented in Section \ref{sec:simulation}, and Section \ref{sec:conclusion} concludes this article.

\vspace{-6pt}
\section{Related Works}
\label{sec:realted-works}
Considerable research efforts have been dedicated to exploring inhomogeneous IPSM. In contrast, there are notably fewer studies focusing on DSG within the scope of time-varying graphs. In the subsequent discussion, we aim to provide an overview of pertinent literature concerning these themes. 

The pioneering work of Hajnal \cite{1958hajnal_bartlett} shows that a Markov chain is weakly ergodic if the set of transition matrices is compact and Markov. Wolfowitz \cite{1963ProdSIAMatrices} extended the result to a category of matrices known as SIA\footnote{A matrix $A$ is said to be stochastic, indecomposable and aperiodic (SIA) if $\lim_{n\rightarrow \infty} A^n = \mathbf{1}c^T$ where $\mathbf{1}$ is the vector of all $1$'s.} matrices, thereby establishing a fundamental connection between weak ergodicity and the left convergence of matrix products. The result of Wolfowitz states that if all finite products of matrices in a given finite set are SIA, then the Markov chain with transition matrices in this set is weakly ergodic. The finiteness condition for the set can be replaced with compactness according to the latter result \cite{1977ExpConvProdStochasticMat}. Since the condition is not easy to check in practice, further studies have focused on $S(\delta)$ stochastic matrices with nonzero diagonal elements. In \cite{2005AgreementConsensusProblems}, a uniform bound $N$ on the matrix product lengths is set such that if every $N$ products of $S(\delta)$ matrices is SIA, then the Markov chain is weakly ergodic. The uniform bound is then removed in \cite{2008CondWeakErg,2016GraphicalDecompCriterion} by introducing disjoint strongly connected components and directed spanning trees. It is worth noting that the method called infinite flow introduced in \cite{2009DSG,2011OnErgodicity,2012OnApproxErg} can be viewed as a graph-level implementation of the property of recurrence of a measure-preserving transformation in the domain of ergodic theory according to Poincaré's recurrence theorem. The most recent works include consideration of generalized stochastic matrices \cite{2020ProdGeneralizedSM} and generalized Sarymsakov matrices \cite{2019GeneralizedSarymsakov}.

Most recent studies of DSG within time-varying graphs depend on dual problem formulation \cite{2020OptimalDistributed,2022OnlineLearningAlgorithm}, proximal gradient algorithms \cite{2021OnlineLearning,2023DistributedProximalGradient} or gradient tracking \cite{2023DistributedContinuousTime,2020AnalysisDesign,2023GTDistributedOptimization}. An underlying assumption of dual problem formulation is that the graph should be undirected and connected to ensure that the graph Laplacian is symmetric and positive semidefinite. The formulation of dual problems naturally introduces minimization problems involving the local objective functions in the generation of iterates, which complicates the update of local iterates especially when local objectives are non-convex. The essence of proximal gradient algorithms lies in dividing the objective function into two parts: a smooth non-convex part and the other non-smooth regular part. Gradient tracking enables agents to reconstruct a progressively estimate of the whole gradient of the local objective functions, while it doubles the computation needed to record an auxiliary variable for each agent. In this article, we will discuss the most general form of DSG which allows further extensions to the aforementioned approaches.

Distinguished from existing studies, our attention is directed towards $S(\delta_t)$ matrices, which allow the nonzero elements to asymptotically tend to $0$ as time increases. We will show, in Section \bigroman{3}, that the consideration of such time-varying lower bounds is necessary for establishing the convergence of the decentralized algorithm with graph-related gradient stretches. 

\begin{figure}[t]
  \centering
  \includegraphics[width=0.48\textwidth]{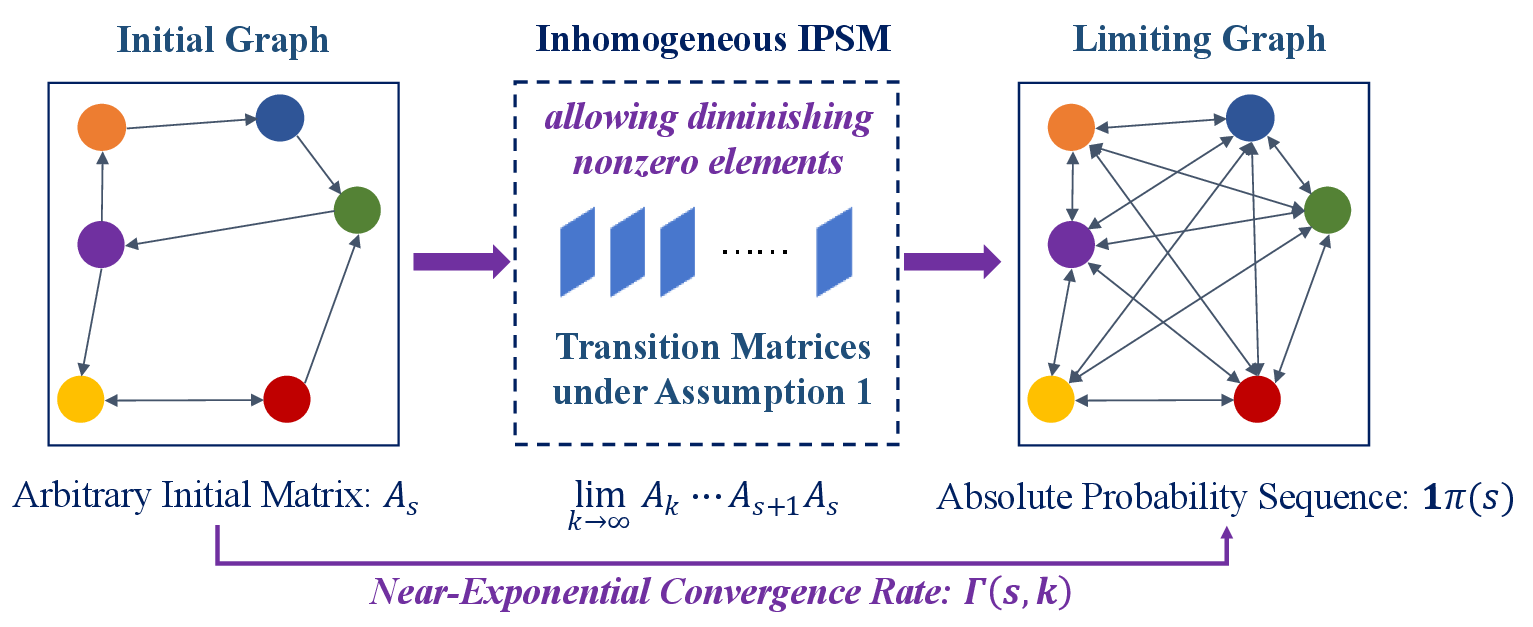}
  \caption{The evolution of inhomogeneous IPSM.}
  \label{fig:system-model}
\end{figure}

%% file: sections/model.tex
\vspace{-6pt}
\section{Theoretical Results on Infinite Products of Stochastic Matrices}
\label{sec:theoretical-results}

Throughout the article, our discussion focuses on a multi-agent system composed of $n$ agents with time-varying topology. We specifically examine scenarios in which the progression of the corresponding weighted directed graph is Markovian, and each transition of the system's topology is representable through a left matrix product. Hence, the asymptotic behavior of the graph is directly linked with the convergence properties of inhomogeneous IPSMs as depicted in Fig. \ref{fig:system-model}. 

\vspace{-6pt}
\subsection{Convergence of Inhomogeneous IPSMs}
To begin with, we introduce the definitions of the Sarymsakov matrix and the scrambling matrix to be used later:
\begin{definition}[Sarymsakov Matrix \cite{1961SarymsakovMatrix}]
  Given a non-negative $r\times r$ matrix $A$ and a set $S \subset \{1,2,\dots,r\}$, the consequent function $F_A(S)$ is given by
  \begin{equation}
    \label{eqn:consequent-function}
    F_A(S) \triangleq \left\{ j: \exists i \in S, \space A_{ij}>0 \right\}.
  \end{equation}
  A non-negative matrix $A$ is called a Sarymsakov matrix if for any two disjoint set $S$ and $S'$,
  \begin{equation}
    \label{eqn:Sarymsakov-condition}
    F_A(S) \cap F_A(S') \neq \emptyset \text{  or  } |F_A(S) \cup F_A(S')| > |S|+|S'|.
  \end{equation}
\end{definition}

\begin{definition}[Scrambling Matrix \cite{1963ProdSIAMatrices}]
  A non-negative matrix $A$ is called scrambling if for any pair of rows $(i,j)$, there is a column $k$ such that $A_{ik} > 0$ and $A_{jk} > 0$.
\end{definition}

Our methodology employed relies on the underlying connections among Sarymsakov matrices, scrambling matrices, and positive-column matrices. In this article, we adopt the following assumption on the transition matrices:
\begin{assumption}
  \label{asp:graph}
  $\{A(t)\}$ is a sequence of $n\times n$ stochastic matrices that satisfy 
  \begin{enumerate}[label=\alph*)]
    \item The diagonal elements of $A(t)$ are positive, i.e., $A_{ii}(t) > 0, \forall i$ for all $t$.
    \item For each $A(t)$, the following statement holds: $\forall S \subsetneq \{1,2,\dots,n\}, \exists (i,j) \in S\times \bar{S}$ such that $A_{ij}(t)>0$.
    \item Denote the shortest communication interval by $B=(n-1)\lceil \log_2 n \rceil$. There exists a positive sequence $\{\beta_t\}, 0 < \beta_t < 1$ and a constant $\delta > 0$ satisfying
    \begin{equation}
      \min_{(s-1)B \le t < sB} \beta_t \ge \left(\frac{\delta}{(s+1)^\lambda}\right)^{1/B}, s\in \mathbb{N}, 0<\lambda<1,
    \end{equation}
    such that $A_{ij}(t) \ge \beta_t$ whenever $A_{ij}(t)>0$ for all $t$. 
  \end{enumerate}
\end{assumption}

\begin{remark}
  Note that Assumption \ref{asp:graph}b) is equivalent to the irreducibility of $A(t)$. To see this, consider the associated digraph. When a stochastic matrix $A$ is irreducible, the associated digraph with respect to $A$ is strongly connected. Hence for any set $S$ of nodes in the graph, there exists an edge from some node in $S$ to a node in $\bar{S}$ and vice versa. As for the other direction, we use contradiction. Under Assumption \ref{asp:graph}b), suppose $A$ is reducible. Then there exists some permutation matrix $P$ such that $PAP^T$ is block upper-triangular with a $K \times (n-K)$ zero block. Denote the permutation with respect to $P$ by $\sigma_P$. Letting $S = \{\sigma_P^{-1}(1), \sigma_P^{-1}(2), \dots, \sigma_P^{-1}(K)\}$, we have $\forall (i,j) \in S \times \bar{S}$, $A_{ij}=0$, which contradicts Assumption \ref{asp:graph}b). Therefore, the equivalence is established.
\end{remark}

To derive further results, a lemma on the convergence rate of $\prod_{t=1}^k (1-x_t)$-type sequence is useful.
\begin{lemma}
  \label{lemma:conv-product-non-summable-sequence}
  Let $\{x_k\} \in \mathbb{R}$ be a non-summable sequence and $0<x_k<1, \forall k$. Then $\prod_{k=1}^\infty (1-x_k) = 0$. Moreover, assuming $\lim_{k\rightarrow \infty} k x_k = \infty$, we obtain
  \begin{enumerate}
    \item $\sum_{k=1}^\infty \prod_{t=1}^k (1-x_t) < \infty$.
    \item $\sum_{k=1}^\infty \sum_{r=k}^\infty \prod_{t=1}^r (1-x_t)<\infty$.
    \item $\lim_{k\rightarrow \infty}k^\mu \sum_{r=k}^\infty \prod_{t=1}^r (1-x_t) = 0$, for any $\mu \in \mathbb{R}$.
  \end{enumerate}
\end{lemma}

\begin{proof}
  Let $y_k = \prod_{t=1}^k (1-x_t)$. By the monotone convergence theorem, there exists $0 \le \varepsilon < 1$ such that $y_k \downarrow \varepsilon$ as $k\rightarrow \infty$. We then prove by contradiction. Suppose $\varepsilon > 0$. We have 
  \begin{equation*}
    \ln \varepsilon \le \ln y_k = \sum_{t=1}^k \ln (1-x_t) \le - \sum_{t=1}^k x_t, \quad \forall k,
  \end{equation*}
  which contradicts the fact that $\{x_k\}$ is non-summable. To show the corresponding series is convergent, we use a standard methodology of real analysis. Denote the Euler-Mascheroni constant by $\gamma$. Since $\lim_{k\rightarrow \infty} kx_k = \infty$, we have for $\varepsilon = 3$ and a constant $F=\gamma+1/\varepsilon$, there exists a natural number $M = M(\varepsilon, F)$ such that $tx_t\ge \varepsilon$ and $|\sum_{i=1}^t 1/i - \ln t| \le F$ for all $t \ge M$. Thus
  \begin{equation*}
    \begin{aligned}
      \sum_{k=M}^\infty \prod_{t=1}^k (1-x_t) &\le \sum_{k=M}^\infty \exp\left\{ - \sum_{t=M}^k x_t \right\} \\
      &= \sum_{k=M}^\infty \exp\left\{ - \sum_{t=M}^k \frac{tx_t}{t} \right\} \\
      &\le \sum_{k=M}^\infty \exp\left\{ \varepsilon (\ln M + 2F - \ln k) \right\} \\
      &= M^3 e^{6F} \sum_{k=M}^\infty k^{-3} < \infty.
    \end{aligned}
  \end{equation*}
  Therefore, we conclude that $\sum_{k=1}^\infty \prod_{t=1}^k (1-x_t) < \infty$. To prove the Part 2), it is sufficient to observe that
  \begin{equation*}
    \sum_{k=1}^{M} \sum_{r=k}^\infty \prod_{t=1}^r (1-x_t) \le M \sum_{k=1}^\infty \prod_{t=1}^k (1-x_t) < \infty,
  \end{equation*}
  and in addition, for all $k \ge M+1$ we have
  \begin{equation}
    \label{eqn:sum-product-argument}
    \begin{aligned}
      \sum_{r=k}^\infty \prod_{t=1}^r (1-x_t) &\le \sum_{r=k}^\infty \exp\left\{ - \sum_{t=M}^{k-1} x_t - \sum_{t=k}^r x_t \right\}\\
      &\le k^3 e^{6F} \sum_{r=k}^\infty r^{-3} \exp\left\{ - \sum_{t=M}^{k-1}x_t \right\}\\
      &\le M^3 e^{12F} \frac{k^2}{(k-1)^4}.
    \end{aligned}
  \end{equation}
  Then we obtain 
  \begin{equation*}
    \sum_{k=M+1}^{\infty} \sum_{r=k}^\infty \prod_{t=1}^r (1-x_t) \le M^3 e^{12F} \sum_{k=M+1}^{\infty} \frac{k^2}{(k-1)^4} < \infty.
  \end{equation*}
  The last part follows the observation that it is sufficient to consider the case where $\mu>0$ and we can always set $\varepsilon = \mu$. By repeating the argument of (\ref{eqn:sum-product-argument}), the proof is completed.
\end{proof}

\begin{corollary}
  \label{col:xtk}
  Let $\{x_k\} \in \mathbb{R}$ be a non-summable sequence and $0<x_k<1, \forall k$ and $\lim_{k\rightarrow \infty} k x_k = \infty$. If $\{y_k\} \in \mathbb{R}^+$ is a summable sequence, we have
  \begin{equation}
    \label{eqn:cor-original-eq}
    \sum_{k=1}^\infty \sum_{r=k}^\infty \prod_{t=1}^r (1-x_{t+k}) y_k < \infty.
  \end{equation}
\end{corollary}

\begin{proof}
  Note that (\ref{eqn:cor-original-eq}) can be rearranged as
  \begin{equation}
    \sum_{k=1}^\infty \sum_{r=k}^\infty \prod_{t=1}^r (1-x_{t+k}) y_k = \sum_{r=1}^\infty \prod_{t=1}^r \sum_{k=r}^\infty y_k(1-x_{t+k}).
  \end{equation}
  It is sufficient to consider the case $x_t \rightarrow 0$. Firstly, for any vanishing positive sequence $\{\beta_{t,r}\}$ with $\beta_{t,r} \rightarrow 0$ as $t\rightarrow \infty$ and $r\ge t$, we have $\sum_{r=1}^\infty \prod_{t=1}^r \beta_{t,r} < \infty$ by Lemma \ref{lemma:conv-product-non-summable-sequence}. The conclusion is straightfoward since we can always choose a sufficiently large $M$ such that $\beta_{t,r} \le 1-x_t$ for all $r \ge t \ge M$. Since $y_k$ is summable, we have for $r\ge t$,
  \begin{equation}
    \limsup_{t \rightarrow \infty} \sum_{k=r}^\infty y_k (1-x_{t+k}) \le \limsup_{r\rightarrow \infty} \sum_{k=r}^\infty y_k = 0.
  \end{equation}
  Hence, it follows that
  \begin{equation}
    \sum_{r=1}^\infty \prod_{t=1}^r \left( \sum_{k=r}^\infty y_k(1-x_{t+k}) \right) < \infty,
  \end{equation}
  which completes the proof.
\end{proof}

% \begin{property}
%   Let $\{A(k)\} \in \mathbb{R}^{n\times n}$ be a sequence of non-negative matrices with positive diagonal entries. Given a permutation $\sigma: \mathbb{N} \rightarrow \mathbb{N}$, a permuted product $\Phi_\sigma(s,k), s \le k$ of $\{A(k)\}$ is defined as 
%   \begin{equation}
%     \Phi_\sigma(s,k) \triangleq A(\sigma(k)) \cdots A(\sigma(s+1)) A(\sigma(s)).
%   \end{equation}
%   Then
%   \begin{enumerate}[label=\alph*)]
%     \item $A_{ij}(t)>0$ $\Longrightarrow$ $[\Phi_\sigma(s,k)A(t)]_{ij}>0, [A(t)\Phi_\sigma(s,k)]_{ij}>0$ for all $i,j, \sigma, t,s,k$, $s \le k$.
%     \item Suppose each $A(t)$ is compliant with a strongly connected graph $\mathscr{G}_t$ without self-loops. There exists a constant $B$ such that $[\Phi_\sigma(s,s+B)]_{ij} > 0$ for all $i,j,s,\sigma$.
%   \end{enumerate}
% \end{property}

Indeed, any matrix satisfying Assumption \ref{asp:graph} is a Sarymsakov matrix based on the following lemma: 
\begin{lemma}
  \label{lemma:Sarymsakov-Matrix}
  If a non-negative $r\times r$ matrix $A$ with positive diagonal entries satisfies that $\forall S\subsetneq \{1,2,\dots,r\}$, $\exists (i,j)\in S\times \bar{S}$ such that $A_{ij}>0$, then $A$ is a Sarymsakov matrix.
\end{lemma}
\begin{proof}
  Consider any two disjoint sets $S$ and $S'$, where $S,S' \subset \{1,2,\dots,r\}$. Since the diagonal entries of $A$ are positive, we have $S \subset F_A(S)$ and hence $|F_A(S)|\ge |S|$ for any $S$. If $F_A(S)\cap F_A(S') \neq \emptyset$, then Condition (\ref{eqn:Sarymsakov-condition}) is satisfied. Otherwise we suppose $F_A(S)\cap F_A(S') = \emptyset$. Since $\forall S \subsetneq \mathcal{V}$, $\exists (i,j)\in S \times \bar{S}$ such that $A_{ij}>0$, we obtain $|F_A(S)|>|S|$. Therefore,  
  \begin{equation*}
    |F_A(S) \cup F_A(S')| = |F_A(S)| + |F_A(S')| > |S|+|S'|,
  \end{equation*}
  which completes the proof.
\end{proof}

The following lemma establishes our main result on the convergence rates of inhomogeneous IPSMs under Assumption \ref{asp:graph}:

\begin{lemma}
  \label{lemma:positive-column}
  Consider a sequence $\{D(k)\}$ of $r\times r$ stochastic matrices for which there exists a positive non-summable sequence $\{\gamma_k\}$ satisfying $\lim_{k\rightarrow \infty} k\gamma_k = \infty$ such that
  \begin{equation}
    \label{eqn:positive-column-condition}
    \forall k, \exists j^*, \text{ s.t. } D_{ij^*}(k) \ge \gamma_k, \quad \forall i.
  \end{equation}
  Then $\lim_{k\rightarrow \infty} D(k)\cdots D(2)D(1)$ exists. Moreover, the convergence rate can be expressed by sum of products of $1-\gamma_k$.
\end{lemma}
\begin{proof}
  Given an arbitrary vector $\mathbf{x} \in \mathbb{R}^r$, we can decompose it into $\mathbf{x} = \mathbf{y}+c\mathbf{1}$, where $c = \min \mathbf{x}^k$ is a scalar representing the minimum of the components of $\mathbf{x}$, and hence $\mathbf{y}$ is a non-negative vector with at least one zero entry. For any $\mathbf{x}_0\in \mathbb{R}^r$, we consider the following iteration for any $s\le k$:
  \begin{equation*}
    \mathbf{x}_k = D(k)\cdots D(2)D(1)\mathbf{x}_0.
  \end{equation*}
  Decompose $\mathbf{x}_k$ into $\mathbf{x}_k=\mathbf{y}_k + c_k \mathbf{1}$. We have
  \begin{equation*}
    \begin{aligned}
      c_{k+1} &= \min_i \mathbf{x}_{k+1}^i = \min_i \sum_{j=1}^r D_{ij}(k+1)\mathbf{x}_k^j \\
      &\ge \min_i \sum_{j=1}^r D_{ij}(k+1)c_k = c_k,
    \end{aligned}
  \end{equation*}
  where the last equality uses the fact that $D(k+1)$ is stochastic. Denoting the index of a zero entry of $\mathbf{y}_{k+1}$ by $p$, we have
  \begin{equation*}
    c_{k+1} = c_k + \sum_{j=1}^r D_{pj}(k+1)\mathbf{y}_k^j.
  \end{equation*} 
  Relying on this property of $\{c_k\}$, we obtain
  \begin{equation*}
    \mathbf{y}_{k+1}^i = \sum_{j=1}^r \left( D_{ij}(k+1)-D_{pj}(k+1) \right) \mathbf{y}_k^j.
  \end{equation*}
  To proceed, we need two useful sets:
  \begin{equation*}
    S_1 = \{j| D_{ij}(k+1)-D_{pj}(k+1)\ge 0\}, S_2 = \bar{S_1}.
  \end{equation*}
  Since $D(k+1)$ is a stochastic matrix, we have
  \begin{equation*}
    \sum_{j \in S_1} D_{ij}(k+1)-D_{pj}(k+1) = \sum_{j\in S_2} D_{pj}(k+1) - D_{ij}(k+1).
  \end{equation*}
  If $j^* \in S_2$, then we obtain
  \begin{equation*}
    \mathbf{y}_{k+1}^i \le \|\mathbf{y}_{k}\|_\infty \sum_{j \in S_1} D_{ij}(k+1) \le (1-\gamma_{k+1}) \|\mathbf{y}_{k}\|_\infty;
  \end{equation*}
  otherwise if $j^* \in S_1$, we consider
  \begin{equation*}
    \mathbf{y}_{k+1}^i \le \|\mathbf{y}_{k}\|_\infty \sum_{j \in S_2} D_{pj}(k+1) \le (1-\gamma_{k+1}) \|\mathbf{y}_{k}\|_\infty,
  \end{equation*}
  where the inequalities hold according to (\ref{eqn:positive-column-condition}). Since $i$ is arbitrary, it follows that
  \begin{equation}
    \label{eqn:yk-iterates}
    \|\mathbf{y}_{k+1}\|_\infty \le (1-\gamma_{k+1}) \|\mathbf{y}_k\|_\infty.
  \end{equation}
  By Lemma \ref{lemma:conv-product-non-summable-sequence}, we conclude that $\mathbf{y}_k \rightarrow \mathbf{0}$. Next we show the convergence of $\{c_k\}$ to some $c=c(\mathbf{x}_0)$. Clearly, we have
  \begin{equation*}
    c_{k+1} \le c_k + \|\mathbf{y}_k\|_\infty.
  \end{equation*}
  Unrolling the iterates, we obtain
  \begin{equation*}
    \begin{aligned}
      c_{k+1} &\le c_0 + \sum_{j=0}^k \|\mathbf{y}_j\|_\infty \\
      &\le c_0 + \|\mathbf{y}_0\|_\infty + \sum_{j=1}^k \prod_{i=1}^j (1-\gamma_i) \|\mathbf{y}_0\|_\infty,
    \end{aligned}
  \end{equation*}
  which by Lemma \ref{lemma:conv-product-non-summable-sequence} is uniformly upper bounded. Hence $\{c_k\}$ is convergent by the monotone convergence theorem, and the limit depends on $\mathbf{x}_0$. In summary, for any $\mathbf{x}_0$, $\mathbf{x}_k$ converges to $c(\mathbf{x}_0)\mathbf{1}$, i.e., $\lim_{k\rightarrow \infty} D(k)\cdots D(2)D(1)$ exists. 

  To derive the convergence rate, it is sufficient to investigate the evolution of $\mathbf{x}_k$ when $\mathbf{x}_0=\mathbf{e}_j$. Denote the backward product of $\{D(t)\}$ from $1$ to $k$ by $\Phi^D(k)$ and represent the limit of $c_k$ by $c^*$ as $k\rightarrow \infty$. We have 
  \begin{equation}
    \|\Phi^D(k)\mathbf{x}_0 - c^*\mathbf{1}\|_\infty \le \|\mathbf{y}_k\|_\infty + c^*-c_k.
  \end{equation}
  According to (\ref{eqn:yk-iterates}) and $\mathbf{x}_0=\mathbf{e}_j$, we obtain
  \begin{equation}
    |\Phi^D_{ij}(k) - c^*| \le \prod_{t=0}^k (1-\gamma_t) + \sum_{t=k}^\infty \prod_{s=0}^t (1-\gamma_s),
  \end{equation}
  which completes the proof.
\end{proof}

This lemma shows that the lower bound on $A(k)$ is transitive through products.
\begin{lemma}
  \label{lemma:backward-product-lower-bound}
  Denote the backward product of $\{A(t)\}$ from $s$ to $k$ by $\Phi(s,k) = A(k)\cdots A(s+1)A(s)$ for any $s\le k$. Under Assumption \ref{asp:graph}c), if $\Phi_{ij}(s,k) > 0$, then $\Phi_{ij}(s,k) \ge \prod_{r=s}^k \beta_r$.
\end{lemma}
\begin{proof}
  We prove by induction. For any $s$, we have $A_{ij}(s) \ge \beta_s$ if $A_{ij}(s) > 0$ by Assumption \ref{asp:graph}c). Next we assume the statement holds for any $k \ge s$. If $\Phi_{ij}(s,k+1) >0$, there exists some $l$ such that $A_{il}(k+1) > 0$ and $\Phi_{lj}(s,k) > 0$. Then by Assumption \ref{asp:graph}c), we have
  \begin{equation*}
    \Phi_{ij}(s,k+1) \ge A_{il}(k+1)\Phi_{lj}(s,k) \ge \prod_{r=s}^{k+1} \beta_r,
  \end{equation*}
  which completes the proof.
\end{proof}

Based on the above lemmas, we conclude that the stochastic matrix sequence $\{A(t)\}$ is strongly ergodic.
\begin{theorem}
  Under Assumption \ref{asp:graph}, the stochastic matrix sequence $\{A(t)\}$ is strongly ergodic.
\end{theorem}
\begin{proof}
  We first show that any inhomogeneous product of $B=(n-1)\lceil \log_2 n\rceil$ matrices in $\{A(t)\}$ has a positive column. This argument is established based on the following three facts: 1) by Lemma \ref{lemma:Sarymsakov-Matrix}, each $A(t)$ is a Sarymsakov matrix; 2) a product of $n-1$ Sarymsakov matrices is scrambling \cite[Section 4]{1979CoefErgodicity}; and 3) a product of $\lceil \log_2 n \rceil$ scrambling matrices has a positive column \cite[Proposition 2.1]{2018ConvergentProdStochasticMatrices}. 

  We continue to consider the limiting behavior of the backward product $\Phi(s+(k-1)B, s+kB-1)$ for any $s$. Denote $\Phi(s+(k-1)B, s+kB-1)$ by $D_s(k)$. We obtain a sequence of stochastic matrices of which each has a positive column. By Lemma \ref{lemma:backward-product-lower-bound}, the elements in the positive column are all larger than or equal to $\gamma_k^s = \prod_{r=s+(k-1)B}^{s+kB-1} \beta_r$. We proceed to show that $\{\gamma_k^s\}$ is non-summable with respect to $k$ and $\lim_{k\rightarrow \infty} k\gamma_k^s = \infty$ for all $s$. Represent $s$ by $s=mB+r, 0\le r<B$. Then based on Assumption \ref{asp:graph}c), we have 
  \begin{subequations}
    \begin{equation*}
      \sum_{k=1}^\infty \gamma_k^s \ge \delta \sum_{k=1}^\infty \frac{1}{(k+m+2)^\lambda} = \infty,
    \end{equation*}
    \begin{equation*}
      \lim_{k\rightarrow \infty} k\gamma_k^s \ge \lim_{k\rightarrow \infty} \frac{\delta k}{(k+m+2)^\lambda} = \infty.
    \end{equation*}
  \end{subequations}
  By Lemma \ref{lemma:positive-column}, $\lim_{k\rightarrow\infty} D_s(k)\cdots D_s(2)D_s(1)$ exists, i.e.,
  \begin{equation*}
    \lim_{k\rightarrow \infty} A(k)\cdots A(s+1)A(s) \text{ exists for all }s, 
  \end{equation*}
  which is the definition of strong ergodicity.
\end{proof}

\begin{definition}[Absolute Probability Sequence \cite{1936OnTheoryofMarkovChains,2017ConvRateWeighted-AveragingDynamics}]
  Let $\{A(t)\}$ be a sequence of stochastic matrices. A sequence of stochastic row vectors $\{\pi(t)\}$ is an absolute probability sequence for $\{A(t)\}$ if
  \begin{equation}
    \pi(t) = \pi(t+1)A(t), \quad \forall t.
  \end{equation}
\end{definition}

It has been shown by Kolmogorov \cite{1936OnTheoryofMarkovChains} that every strongly ergodic sequence $\{A(t)\}$ has an absolute probability sequence $\{\pi(t)\}$ with non-negative elements, i.e.,
\begin{equation}
  \lim_{k\rightarrow \infty} A(k)\cdots A(t+1)A(t) = \mathbf{1}\pi(t), \quad \forall t.
\end{equation}
Based on the absolute probability sequence, it is natural to infer through Lemma \ref{lemma:positive-column} that the following proposition holds:
\begin{proposition}
  \label{prop:convergence-rates-backward-product}
  Denote the backward product of $\{A(t)\}$ from $s$ to $k$ by $\Phi(s,k) = A(k)\cdots A(s+1)A(s)$ for any $s\le k$. Under Assumption \ref{asp:graph}, the convergence rate is given by
  \begin{equation}
    \left| \Phi_{ij}(s,k) - \pi_j(s) \right| \le \Gamma(s,k), \quad \forall i,j,
  \end{equation}
  where $\Gamma(s,k)$ is expressed as
  \begin{equation}
    \label{eqn:Gamma-definition}
    \Gamma(s,k) = \prod_{t=0}^{\lfloor \frac{k-s}{B} \rfloor} (1-\gamma_t^s) + \sum_{t=\lfloor \frac{k-s}{B} \rfloor}^\infty \prod_{r=0}^t 
    (1-\gamma_r^s),
  \end{equation}
  and $\gamma_t^s$ is written as
  \begin{equation}
    \gamma_t^s = \frac{\delta}{(\lfloor s/B \rfloor + t + 1)^\lambda}.
  \end{equation}
  Moreover, we define $\Gamma(l,k)=\Gamma(k,k)$ if $l>k$.
\end{proposition}
Next we present a lemma on a useful property of $\Gamma(s,k)$. The following lemma reveals the fact that $\Gamma(s,k)$ shares some similarities in terms of the limiting behavior with an exponential function of the form $\beta^{s-k}$, where $\beta$ is a constant. 

\begin{lemma}
  \label{lemma:conv-Gamma}
  Let $\Gamma(s,k)$ be given by (\ref{eqn:Gamma-definition}), then
  \begin{enumerate}
    \item $\forall \mu > 0, s$, $\lim_{k\rightarrow \infty} k^\mu \Gamma(s,k) = 0$;
    \item Letting $\{\alpha_t\} \in \mathbb{R}$ be a positive sequence and $\alpha_t \rightarrow 0$ as $t\rightarrow \infty$, we have $\lim_{k\rightarrow \infty} \sum_{t=1}^k \alpha_t \Gamma(t,k) = 0$;
    \item Let $\{\alpha_t\} \in \mathbb{R}$ be a positive and square summable sequence. We have $\sum_{k=1}^\infty \alpha_k \sum_{t=1}^k \alpha_t \Gamma(t,k) < \infty$.
  \end{enumerate}
\end{lemma}
\begin{proof}
  Part 1) is a direct result of Lemma \ref{lemma:conv-product-non-summable-sequence}. We focus on the proof of Part 2) and Part 3). 

  \textbf{Proof of Part 2)}: Since $\lim_{k\rightarrow\infty} \alpha_k=0$, we have $\forall \varepsilon > 0$, and there exists some $M$ such that $\alpha_k \le \varepsilon$ for any $k\ge M$. Then $\forall k > M$,
  \begin{equation*}
    \sum_{t=1}^k \alpha_t \Gamma(t,k) \le \sum_{t=1}^M \alpha_t \Gamma(t,k) + \varepsilon \sum_{t=M+1}^k \Gamma(t,k) .
  \end{equation*}
  Since $\lim_{k\rightarrow \infty} \Gamma(t,k) = 0$ for all $t$, we obtain
  \begin{equation*}
    \limsup_{k\rightarrow\infty} \sum_{t=1}^M \alpha_t \Gamma(t,k) \le M \sup_{1\le t\le M} \alpha_t \limsup_{k\rightarrow\infty} \Gamma(t,k) = 0.
  \end{equation*}
  To derive the second part, it is sufficient to verify that $\sum_{t=M+1}^k \sum_{s=\lfloor \frac{k-t}{B} \rfloor}^\infty \prod_{r=0}^s (1-\gamma_r^t)$ is uniformly bounded $\forall k$. To establish this, we have
  \begin{equation}
    \label{eqn:conv-sum-Gamma}
    \begin{aligned}
      \sum_{t=M+1}^k \sum_{s=\lfloor \frac{k-t}{B} \rfloor}^\infty \gamma_{(s)}^t &\le \sum_{m=0}^{\lfloor \frac{k-1-M}{B} \rfloor} \sum_{t=k-1-(m+1)B}^{k-1-mB} \sum_{s=m}^\infty \gamma_{(s)}^t \\
      &\le B\limsup_{k\rightarrow \infty} \sum_{m=0}^k \sum_{s=m}^\infty \gamma_{(s)}^{k-1-(m+1)B} \\
      &< \infty,
    \end{aligned}
  \end{equation}
  where $\gamma_{(s)}^t$ is short for $\prod_{r=0}^s (1-\gamma_r^t)$, and the last inequality holds by Lemma \ref{lemma:conv-product-non-summable-sequence}. Since $\varepsilon$ is selected arbitrarily, we have
  \begin{equation*}
    \limsup_{k\rightarrow \infty} \sum_{t=1}^k \alpha_t \Gamma(t,k) = 0 \Longrightarrow \lim_{k\rightarrow \infty} \sum_{t=1}^k \alpha_t \Gamma(t,k) = 0,
  \end{equation*}
  which completes the proof of Part 2).

  \textbf{Proof of Part 3)}: Based on the Cauchy-Schwarz inequality, we obtain
  \begin{equation*}
    \sum_{k=1}^\infty \alpha_k \sum_{t=1}^k \alpha_t \Gamma(t,k) \le \sum_{k=1}^\infty \alpha_k^2 \sum_{t=1}^k \Gamma(t,k) + \sum_{k=1}^\infty \sum_{t=1}^k \alpha_t^2 \Gamma(t,k).
  \end{equation*}
  According to (\ref{eqn:conv-sum-Gamma}), we conclude that $\sum_{t=1}^k \Gamma(t,k)$ is uniformly bounded $\forall k$. Since $\{\alpha_k\}$ is square summable, we have
  \begin{equation*}
    \sum_{k=1}^\infty \alpha_k^2 \sum_{t=1}^k \Gamma(t,k) < \infty.
  \end{equation*}
  In addition, we have by Corollary \ref{col:xtk}
  \begin{equation}
    \sum_{k=1}^\infty \sum_{t=1}^k \alpha_t^2 \Gamma(t,k) = \sum_{t=1}^\infty \alpha_t^2 \sum_{k=t}^\infty \Gamma(t,k) < \infty,
  \end{equation}
  which completes the proof. 
\end{proof}

\vspace{-6pt}
\subsection{Limit of the Absolute Probability Sequence}
In general, there is no more explicit conclusion on the evolution of the absolute probability sequence $\{\pi(k)\}_{k\in \mathbb{N}^+}$ without additional assumptions. In this work, we provide insight into the asymptotic behavior of $\pi(k)$ when the stochastic matrices $\{A(t)\}$ tend to an identity matrix as stated in the following assumption:
% Furthermore, to control the potential impacts induced by the factor $1/A_{ii}(k)$, a restriction on the asymptotic behavior of $A(k)$ is necessary.
\begin{assumption}
  \label{asp:conv-time-varying-matrix}
  \begin{enumerate}
    \item The stochastic matrix sequence tends to the identity matrix with rate $\mathcal{O}(1/\ln k)$, i.e.,
    \begin{equation}
      \sum_{k=1}^\infty \frac{1}{k} \|I-A(k)\| < \infty.
    \end{equation}
    Therefore, $\left(\sum_{i=1}^n 1/A_{ii}(k)\right)^2$ is uniformly bounded by some constant $\chi^2$ for all $k$.
    \item There exists a natural number $N>1$ and a positive real sequence $\{\omega_k\}_{k\in \mathbb{N}^+}$ satisfying $\sum_{k=1}^\infty \omega_k \ln k < \infty$ such that
    \begin{equation}
      1 - \Phi_{ii}(k,k+N-1) \le \omega_k
    \end{equation}
    holds $\forall k$. Therefore, it follows that
    \begin{equation}
      \Phi_{ij}(k,k+N-1) \le \omega_k, \quad \forall k, i\neq j.
    \end{equation}
  \end{enumerate}
\end{assumption}

\begin{remark}
  Indeed, it is straightfoward to verify that Assumption \ref{asp:conv-time-varying-matrix} conflicts with any $S(\delta)$ matrix sequence for a fixed $\delta > 0$. In addition, it can be conluded that 
  \begin{equation}
    \label{eqn:lnk-sum-dominate}
    \sum_{n=1}^\infty \sum_{m=n}^\infty \frac{1}{n} \omega_m = \sum_{k=1}^\infty \sum_{n=1}^k \frac{1}{n} \omega_k < \infty.
  \end{equation}
  This follows directly the fact that the growth of the harmonic series is asymptotic to $\mathcal{O}(\ln k)$.
\end{remark} 

Before proceeding to reveal the asymptotic behavior of $\pi(k)$ under Assumption \ref{asp:conv-time-varying-matrix}, we need the following simple lemma:
\begin{lemma}
  \label{lemma:conv-1-xt-summable}
  Letting $\{x_t\}$ be a summable positive real sequence with $0\le x_t <1, \forall t$, we have $\lim_{s\rightarrow \infty} \prod_{k=0}^\infty (1-x_{s+k}) = 1$. If further assuming $\sum_{k=1}^\infty x_k \ln k < \infty$, we have 
  \begin{equation}
    \sum_{n=1}^\infty \frac{1}{n} \left[ 1- \prod_{k=0}^\infty (1-x_{n+k}) \right] < \infty.
  \end{equation}
\end{lemma}
\begin{proof}
  As $\{x_t\}$ is summable and $1-x_t \neq 0$, we have $\prod_{k=1}^\infty (1-x_k) > 0$. Since $\prod_{k=0}^\infty (1-x_{s+k}) = c_s$ for all $s$ and $1 \ge c_{s+1} > c_s > 0$, $\{c_s\}$ converges as $s \rightarrow \infty$ according to the monotone convergence theorem. Therefore, we have 
  \begin{equation}
    c_s = \frac{c_1}{\prod_{t=1}^s (1-x_t)}.
  \end{equation}
  Taking $s\rightarrow \infty$ on both sides, we obtain $\lim_{s\rightarrow \infty} c_s =1$.
  
  To proceed, we assume $x_t \le \frac{1}{2}$ for all $t$ without loss of generality (there are only finitely many points $x_k > 1/2$). For any $0 \le x \le \frac{1}{2}$, we have $e^{-2x} \le 1-x$. Hence,
  \begin{equation}
    1 - \prod_{k=0}^\infty (1-x_{n+k}) \le 1 - e^{-2\sum_{k=0}^\infty x_{n+k}}.
  \end{equation}
  Since $\lim_{x\rightarrow 0} \frac{1-e^{-x}}{x} = 1$, it is sufficient to consider 
  \begin{equation}
    \sum_{n=1}^\infty \sum_{k=0}^\infty \frac{1}{n} x_{n+k} = \sum_{n=1}^\infty \sum_{k=n}^\infty \frac{1}{n} x_k < \infty,
  \end{equation}
  which holds by a derivation similar to that for (\ref{eqn:lnk-sum-dominate}). 
\end{proof}

\begin{theorem}
  \label{theorem:conv-absolute-prob-seq}
  Under Assumptions \ref{asp:graph} and \ref{asp:conv-time-varying-matrix}, the absolute probability sequence $\{\pi_k\}$ tends to the uniform distribution $\frac{1}{n} \mathbf{1}^T$ and the convergence rate satisfies
  \begin{equation}
    \sum_{k=1}^\infty \frac{1}{k} \left|\pi_i(k) - \frac{1}{n}\right| < \infty, \quad i=1,2,\dots,n.
  \end{equation}
\end{theorem} 
\begin{proof}
  According to Assumption \ref{asp:conv-time-varying-matrix}, it can be directly seen that $\Phi(k,k+N-1)$ is a strictly diagonally dominant stochastic matrix for sufficiently large $k$, and hence invertible. Here we apply a trick to decompose the matrix into
  \begin{equation}
    \Phi(k,k+N-1) = (1-\omega_k) I + \omega_k \phi_k,
  \end{equation}
  where $\phi_k$ is a stochastic matrix. Consider a matrix norm $\|A\|_\infty = \max_{i,j} |A_{ij}|$ such that $\|I\|_\infty = 1$. By the Banach lemma, we obtain
  \begin{equation}
    \begin{aligned}
      \|\Phi^{-1}(k,k+N-1)\|_\infty &= \|((1-\omega_k) I + \omega_k \phi_k)^{-1}\|_\infty \\
      &\le \frac{1-\omega_k}{1 - \omega_k \|\phi_k\|_\infty / (1-\omega_k)} \\
      &\le 1 + \frac{\omega_k^2}{1-2\omega_k},
    \end{aligned}
  \end{equation}
  for sufficiently large $k$. Moreover,
  \begin{equation}
    \begin{aligned}
      &\|\Phi^{-1}(k,k+N-1) - I\|_\infty \\ 
      &\le \|\Phi^{-1}(k,k+N-1)\|_\infty \|\Phi(k,k+N-1) - I\|_\infty \\
      &\le \omega_k \left( 1+ \frac{\omega_k^2}{1-2\omega_k} \right).
    \end{aligned}
  \end{equation}
  Therefore, we conclude that, for sufficiently large $k$,
  \begin{subequations}
    \begin{equation}
      \Phi^{-1}_{ii}(k,k+N-1) \le 1 + \frac{\omega_k^2}{1-2\omega_k}, \forall i, 
    \end{equation}
    \begin{equation}
      \Phi^{-1}_{ij}(k,k+N-1) \le \omega_k \left( 1+ \frac{\omega_k^2}{1-2\omega_k} \right), \forall i\neq j.
    \end{equation}
  \end{subequations}
  For any $k$ and sufficiently large $m \in \mathbb{N}$, we have
  \begin{equation}
    \pi(k_{m+1}^N) = \pi(k_m^N) \Phi^{-1}(k_m^N,k_{m+1}^N-1),
  \end{equation}
  where $k_m^N$ is short for $k+mN$. Hence for any $i$,
  \begin{equation*}
    \pi_i(k_{m+1}^N) \le \pi_i(k_m^N)\left(1+\frac{\omega_k^2}{1-2\omega_k}\right) + n\omega_k \left(1+\frac{\omega_k^2}{1-2\omega_k}\right).
  \end{equation*}
  By the convergence theorem for supermartingales \cite{1971SuperMartingale}, we conclude that $\pi_i(k_m^N)$ converges as $m\to \infty$, i.e.,
  \begin{equation}
    \lim_{m\rightarrow \infty} \pi_i(k_m^N) = \theta_i^k, \quad \forall i,k.
  \end{equation}
  Moreover, we have $\theta_i^k = \theta_i^{k+N}$, which means that the limit is a periodic function with respect to $k$. We observe that
  \begin{equation}
    \begin{aligned}
      \pi_i(k_m^N) &= \sum_{j=1}^n \pi_j(k_m^N+1) A_{ji}(k_m^N) \\
      &\le \pi_i(k_m^N+1) A_{ii}(k_m^N) + \sum_{j\neq i} A_{ji}(k_m^N).
    \end{aligned}
  \end{equation}
  Taking $m\to \infty$ at both sides, we obtain by Assumption \ref{asp:conv-time-varying-matrix}
  \begin{equation}
    \theta_i^k \le \theta_i^{k+1}, \quad \forall k \in \mathbb{N}^+.
  \end{equation}
  Therefore, it follows that for any $k$
  \begin{equation}
    \theta_i^k \le \theta_i^{k+1} \le \cdots \le \theta_i^{k+N},
  \end{equation}
  which indicates that $\theta_i^k = \theta_i$ as $\theta_i^k=\theta_i^{k+N}$. 
  
  Next we show that $\theta_i$ is independent of $i$ by applying Proposition \ref{prop:convergence-rates-backward-product}. We consider the difference between $\pi_i(k)$ and $\pi_j(k)$, $\forall i\neq j$ as follows:
  \begin{equation}
    \begin{aligned}
      |\pi_i(k) - \pi_j(k)| \le |\Phi_{ii}(k,k_m^N) -1| + |1 - \Phi_{jj}(k,k_m^N)| \\ + |\Phi_{ii}(k,k_m^N) - \pi_i(k)| + |\Phi_{jj}(k,k_m^N) - \pi_j(k)|.
    \end{aligned}
  \end{equation}
  Taking $m\rightarrow \infty$ at both sides, we obtain by Lemma \ref{lemma:backward-product-lower-bound}
  \begin{equation}
    \label{eqn:pik-diff}
    |\pi_i(k) - \pi_j(k)| \le 2 \left[ 1- \prod_{m=0}^\infty (1-\omega_{k+mN}) \right].
  \end{equation}
  Taking $k\to \infty$ at both sides, it follows that by Lemma \ref{lemma:conv-1-xt-summable}
  \begin{equation}
    |\theta_i - \theta_j| = 0.
  \end{equation}
  Since $\sum_{i=1}^n \theta_i = 1$, we conclude that $\theta_i = \frac{1}{n}$, $\forall i$. If $\pi(k) \neq \frac{1}{n} \mathbf{1}^T$, then $\forall p$ with $\pi_p(k) > 1/n$ there must be some $q \neq p$ such that $\pi_p(k) > 1/n > \pi_q(k)$. By (\ref{eqn:pik-diff}),
  \begin{equation}
    \begin{aligned}
      \pi_p(k) - \frac{1}{n} &< \pi_p(k) - \frac{1}{n} + \frac{1}{n} - \pi_q(k) \\
      &= \pi_p(k) - \pi_q(k) \\
      &\le 2 \left[ 1- \prod_{m=0}^\infty (1-\omega_{k+mN}) \right].
    \end{aligned}
  \end{equation}
  The same argument still holds for any $p$ with $\pi_p(k)<1/n$. Therefore, it follows $\forall i$ that
  \begin{equation}
    \left|\pi_i(k) - \frac{1}{n}\right| \le 2\left[ 1- \prod_{m=0}^\infty (1-\omega_{k+mN}) \right].
  \end{equation}
  By Lemma \ref{lemma:conv-1-xt-summable}, we conclude that
  \begin{equation}
    \sum_{k=1}^\infty \frac{1}{k} \left| \pi_i(k) - \frac{1}{n} \right| < \infty, \quad i=1,2,\dots,n,
  \end{equation}
  which completes the proof.
\end{proof}

%% file: sections/application.tex
\vspace{-6pt}
\section{Application to Subgradient Methods}
\label{sec:app}
The standard distributed optimization is formulated as finding the solution to the following problem:
\begin{equation}
  \label{eqn:optimization-objective}
  \min_{\mathbf{x} \in \mathcal{X}} f(\mathbf{x}) = \sum_{i=1}^n f_i(\mathbf{x}),
\end{equation}
where each $f_i$ is a private objective function only available to agent $i$, and $\mathcal{X}$ is a non-empty closed convex set bounded by $\eta$, i.e., $\sup_{\mathbf{x},\mathbf{y} \in \mathcal{X}} \|\mathbf{x} - \mathbf{y}\|\le \eta$. We assume that the solution set $\mathcal{X}^*$ of (\ref{eqn:optimization-objective}) is not empty. 

In practical multi-agent networks, it is not always feasible to transform a weighted graph into a doubly-stochastic matrix. Therefore, the problem of unbalanced graphs arises. 

% In pratical multi-agent networks, the raw data vectors are compressed or encoded locally before transmission to reduce transmission overheads and eliminate potential security concerns, which naturally induces data errors.

In this section, we propose the following unbalanced distributed projected subgradient (UDPSG) method to cope with cases of unbalanced weighted graphs where the time-varying matrix sequence $\{A(t)\}$ satisfies Assumptions \ref{asp:graph} and \ref{asp:conv-time-varying-matrix}: 
\begin{equation}
  \label{eqn:unconstrained-iterates}
  \mathbf{x}_{i,k+1} = \mathcal{P}_{\mathcal{X}}\left[ \sum_{j=1}^n A_{ij}(k) \mathbf{x}_{j,k} - \alpha_k \frac{\mathbf{g}_{i,k}}{A_{ii}(k)} \right],
\end{equation}
where $\mathbf{x}_{i,k} \in \mathbb{R}^m$ is the raw data vector of agent $i$ maintained at iteration $k$, $\mathbf{g}_{i,k} \in \partial f_i(\mathbf{x}_{i,k})$ represents the subgradient of $f_i$ evaluated at $\mathbf{x}_{i,k}$, and the positive real sequence $\{\alpha_k\}$ is not summable but square summable. Represent the matrix composed of stacked vectors $\{\mathbf{x}_{i,k}^T\} \in \mathbb{R}^m$ and $\{\mathbf{g}_{i,k}^T / A_{ii}(k)\} \in \mathbb{R}^m$ by $X(k) \in \mathbb{R}^{n\times m}$ and $G(k) \in \mathbb{R}^{n\times m}$, respectively. Denote the projection error by $\xi_{i,k}$ which is written as
\begin{equation}
  \label{eqn:xik-projection-error}
  \xi_{i,k} = \mathbf{x}_{i,k+1} - \mathbf{v}_{i,k}, 
\end{equation}
where $\mathbf{v}_{i,k}$ is expressed by
\begin{equation}
  \label{eqn:vik}
  \mathbf{v}_{i,k} = \sum_{j=1}^n A_{ij}(k)\mathbf{x}_{j,k} - \alpha_k \frac{\mathbf{g}_{i,k}}{A_{ii}(k)}.
\end{equation}
Denoting the projection error matrix by $\Xi(k)$, we obtain
\begin{equation}
  \label{eqn:X-matrix-iterates}
  X(k+1) = A(k)X(k) - \alpha_k G(k) + \Xi(k).
\end{equation}
The following result on projection errors will be helpful:
\begin{lemma}[Lemma 1, \cite{2010ConstrainedConsensus}]
  \label{lemma:projection-errors}
  Let $\mathcal{X}$ be a nonempty closed convex set in $\mathbb{R}^d$. Then we have $\forall \mathbf{x} \in \mathbb{R}^d$ and $\mathbf{y} \in \mathcal{X}$
  \begin{subequations}
    \begin{equation}
      \langle \mathcal{P}_\mathcal{X} [\mathbf{x}] - \mathbf{x}, \mathbf{x} - \mathbf{y} \rangle \le - \|\mathcal{P}_\mathcal{X} [\mathbf{x}] - \mathbf{x}\|^2,
    \end{equation}
    \begin{equation}
      \|\mathcal{P}_\mathcal{X} [\mathbf{x}] - \mathbf{y}\|^2 \le \|\mathbf{x} - \mathbf{y}\|^2 - \|\mathcal{P}_\mathcal{X} [\mathbf{x}] - \mathbf{x}\|^2.
    \end{equation}
  \end{subequations}
\end{lemma}
Based on this lemma, we obtain
\begin{equation}
  \label{eqn:estimation-projection-error}
  \begin{aligned}
    \mathbb{E}\|\Xi(k)\|^2 &= \sum_{i=1}^n \mathbb{E}\|\xi_{i,k}\|^2 \\ 
    &\le \sum_{i=1}^n \mathbb{E} \left\|\alpha_k \frac{\mathbf{g}_{i,k}}{A_{ii}(k)} \right\|^2 \le \sum_{i=1}^n \alpha_k^2 \frac{\mathbb{E}\|\mathbf{g}_{i,k}\|^2}{(A_{ii}(k))^2}.
  \end{aligned} 
\end{equation}
Unrolling the iterates (\ref{eqn:X-matrix-iterates}), we obtain $\forall s\le k$,
\begin{equation*}
  X(k+1) = \Phi(s,k) X_s - \sum_{t=s}^k  \Phi(t+1,k) (\alpha_t G(t) - \Xi(t)),
\end{equation*}
where we define $\Phi(k+1,k)=I$ for convenience. To study the limiting behavior of $\{\mathbf{x}_{i,k}\}$, we consider an auxiliary sequence $\{\mathbf{y}_k\}$ with the initial value as
\begin{equation}
  \mathbf{y}_0 = \sum_{i=1}^n \pi_i(0) \mathbf{x}_{i,0},
\end{equation}
and the sequence updates according to
\begin{equation}
  \label{eqn:y-iterates}
  \mathbf{y}_{k+1} =  \mathbf{y}_k + \sum_{i=1}^n \pi_i(k+1) \mathbf{u}_{i,k},
\end{equation}
where $\mathbf{u}_{i,k} = \mathbf{v}_{i,k} - \sum_{j=1}^n A_{ij}(k) \mathbf{x}_{j,k} + \xi_{i,k}$. 

Adopting a similar notation to $X(k)$ and $\Xi(k)$, we have $Y(k)=\mathbf{1}\mathbf{y}_k^T$. Furthermore, a common assumption on the boundedness of subgradients (equivalent to Lipschitz continuity) is introduced as follows:
\begin{assumption}
  \label{asp:bounded-subgradients}
  For all $i$, all subgradients of $f_i$ are uniformly bounded by $L$, i.e., $\|\mathbf{g}\|\le L$ for all $\mathbf{g} \in \partial f_i$, $\forall i$.
\end{assumption}
Under Assumptions \ref{asp:graph}-\ref{asp:bounded-subgradients}, we have the following lemma on the association between $\mathbf{x}_{i,k}$ and $\mathbf{y}_k$.
\begin{lemma}
  \label{lemma:conv-xk-yk}
  Let $\{\mathbf{x}_{i,k}\}$ be generated by (\ref{eqn:unconstrained-iterates}) and $\{\varrho_k\}$ be a square summable positive real sequence. Then under Assumptions \ref{asp:graph}-\ref{asp:bounded-subgradients}, it follows that
  \begin{enumerate}[label=(\alph*)]
    \item $\lim_{k\rightarrow \infty} \mathbb{E}\|\mathbf{x}_{i,k} - \mathbf{y}_k\| = 0$, $\forall i$;
    \item $\sum_{k=1}^\infty \varrho_k \mathbb{E}\|\mathbf{x}_{i,k} - \mathbf{y}_k\| < \infty$, $\forall i$;
    \item $\sum_{k=1}^\infty \varrho_k \mathbb{E}\|\mathbf{x}_{i,k+1} - \mathbf{x}_{i,k}\| < \infty$, $\forall i$.
  \end{enumerate}
\end{lemma}
\begin{proof}
  \textbf{Proof of Part (a)}: According to the definition of $X(k)$ and $Y(k)$, we have
  \begin{equation*}
    \begin{aligned}
      X(k+1) - Y(k+1) = (\Phi(0,k) - \mathbf{1}\pi(0))X(0) \\ - \sum_{t=1}^k (\Phi(t+1,k)-\mathbf{1}\pi(t+1)) (\alpha_t G(t) - \Xi(t)).
    \end{aligned}
  \end{equation*}
  By Proposition \ref{prop:convergence-rates-backward-product}, we obtain
  \begin{equation}
    \|(\Phi(0,k) - \mathbf{1}\pi(0))X(0)\| \le n \Gamma(0,k) \|X(0)\|_\infty,
  \end{equation}
  and by Assumption \ref{asp:bounded-subgradients} and (\ref{eqn:estimation-projection-error}), it follows that
  \begin{equation}
    \mathbb{E}\|(\Phi(t+1,k)-\mathbf{1}\pi(t+1))G(t)\| \le n \Gamma(t+1,k) \chi L,
  \end{equation}
  \begin{equation}
    \mathbb{E}\|(\Phi(t+1,k)-\mathbf{1}\pi(t+1))\Xi(t)\| \!\le\! 2\alpha_t \chi L n\Gamma(t+1,k).
  \end{equation}
  According to Lemma \ref{lemma:conv-Gamma} item 2) and Assumption \ref{asp:conv-time-varying-matrix}, we have
  \begin{equation}
    \lim_{k\rightarrow \infty} \mathbb{E}\|X(k+1)-Y(k+1)\| = 0,
  \end{equation}
  which indicates $\lim_{k\rightarrow \infty} \mathbb{E}\|\mathbf{x}_{i,k} - \mathbf{y}_k\| = 0$, $\forall i$.

  \textbf{Proof of Part (b)}: By Lemma \ref{lemma:conv-Gamma} item 3) and Assumption \ref{asp:conv-time-varying-matrix}, we obtain
  \begin{subequations}
    \begin{equation}
      \sum_{k=0}^\infty \varrho_{k+1}\Gamma(0,k) < \infty,
    \end{equation}
    \begin{equation}
      \sum_{k=0}^\infty \varrho_{k+1} \sum_{t=1}^k \varrho_t \Gamma(t+1,k) < \infty.
    \end{equation}
  \end{subequations}
  Therefore, it follows that
  \begin{equation}
    \sum_{k=1}^\infty \varrho_{k} \mathbb{E}\|X(k)-Y(k)\| < \infty.
  \end{equation}
  Hence we conclude that $\sum_{k=1}^\infty \varrho_k \mathbb{E}\|\mathbf{x}_{i,k} - \mathbf{y}_k\| < \infty$, $\forall i$.

  \textbf{Proof of Part (c)}: From the proof of Part (a) and Part (b), we conclude that $\sum_{k=1}^\infty \varrho_k \|\mathbf{x}_{i,k+1} - \mathbf{y}_{k+1}\| < \infty$. Moreover,
  \begin{equation}
    \mathbb{E}\|\mathbf{y}_{k+1} - \mathbf{y}_k\| \le \sum_{i=1}^n \mathbb{E}\|\mathbf{u}_{i,k}\| 3\alpha_k L \chi.
  \end{equation}
  By dividing $\|\mathbf{x}_{i,k+1}-\mathbf{x}_{i,k}\|$ into three parts, i.e., $\|\mathbf{x}_{i,k+1}- \mathbf{y}_{k+1}\|$, $\|\mathbf{y}_{k+1} - \mathbf{y}_{k}\|$, and $\|\mathbf{x}_{i,k} - \mathbf{y}_k\|$, we obtain $\sum_{k=1}^\infty \varrho_k \mathbb{E}\|\mathbf{x}_{i,k+1} - \mathbf{x}_{i,k}\| < \infty$.
\end{proof}
In the following analysis, we restrict our discussion to the case in which $\alpha_k \in \mathcal{O}(1/k)$. 

\vspace{-8pt}
\subsection{Convex Local Objective Functions}
We proceed to reveal the iteration relation between $\mathbf{y}_{k+1}$ and $\mathbf{y}_k$ which is critical to the convergence of $\{\mathbf{y}_k\}$, regarding the cases where each $f_i$ is convex. 
\begin{lemma}
  Under Assumptions \ref{asp:graph}-\ref{asp:bounded-subgradients}, for any vector $\mathbf{x} \in \mathcal{X}$, the following iteration holds:
  \begin{equation*}
    \begin{aligned}
      \mathbb{E}\|\mathbf{y}_{k+1}-\mathbf{x}\|^2 \le \mathbb{E}\|\mathbf{y}_k-\mathbf{x}\|^2 \! +\! \frac{2\alpha_k}{n}\mathbb{E}[f(\mathbf{x}) - f(\mathbf{y}_k)]\\
      + 4\zeta_k^2 + 4\zeta_k \sum_{i=1}^n \left( 2\mathbb{E}\|\mathbf{y}_k - \mathbf{x}_{i,k}\| + \mathbb{E}\|\mathbf{x}_{i,k+1} - \mathbf{x}_{i,k}\| \right)\\
      + 2 \zeta_k \! \left[ \eta + \sum_{i=1}^n \|\mathbf{y}_k - \mathbf{x}_{i,k}\| \right]  \left| \frac{\pi_i(k+1)}{A_{ii}(k)} - \frac{1}{n} \right|,
    \end{aligned}
  \end{equation*}
  where $\zeta_k = \alpha_k \chi L$. Furthermore, it follows that
  \begin{equation}
    \label{eqn:finiteness-of-f-dev}
    \sum_{k=1}^\infty \alpha_k \mathbb{E}\left[f(\mathbf{y}_k) - f(\mathbf{x})\right] < \infty.
  \end{equation}
\end{lemma}
\begin{proof}
  Unrolling the iterates of $\mathbf{y}_k$, we obtain
  \begin{equation}
    \mathbb{E}\|\mathbf{y}_{k+1} - \mathbf{y}_{k}\|^2 = \mathbb{E}\left\|\sum_{i=1}^n \pi_i(k+1) \mathbf{u}_{i,k} \right\|^2 \le 4\alpha_k^2 L^2 \chi^2.
  \end{equation}
  Moreover, by Lemma \ref{lemma:projection-errors} it follows that
  \begin{subequations}
    \begin{equation}
      \mathbb{E}\langle \xi_{i,k}, \mathbf{y}_k-\mathbf{x}_{i,k} \rangle \le \mathbb{E}\|\xi_{i,k}\| \mathbb{E}\|\mathbf{y}_k-\mathbf{x}_{i,k}\|,
    \end{equation}
    \begin{equation}
      \mathbb{E}\langle \xi_{i,k}, \mathbf{x}_{i,k}-\mathbf{x}_{i,k+1} \rangle \le \mathbb{E}\|\xi_{i,k}\| \mathbb{E}\|\mathbf{x}_{i,k}-\mathbf{x}_{i,k+1}\|,
    \end{equation}
    \begin{equation}
      \mathbb{E}\langle \xi_{i,k}, \mathbf{x}_{i,k+1}-\mathbf{v}_{i,k} \rangle = \mathbb{E}\|\xi_{i,k}\|^2,
    \end{equation}
    \begin{equation}
      \mathbb{E}\langle \xi_{i,k}, \mathbf{v}_{i,k}-\mathbf{x} \rangle \le - \mathbb{E}\|\xi_{i,k}\|^2.
    \end{equation}
  \end{subequations}
  Combining the equations above, we obtain
  \begin{equation}
    \mathbb{E}\langle \xi_{i,k}, \mathbf{y}_k -\mathbf{x}\rangle \le \mathbb{E}\|\xi_{i,k}\| \mathbb{E}(\|\mathbf{y}_k-\mathbf{x}_{i,k}\| + \|\mathbf{x}_{i,k}-\mathbf{x}_{i,k+1}\|).
  \end{equation}
  Next we analyze the critical term which contains the factor $1/A_{ii}(k)$. First, it follows directly from the boundedness of the subgradients that
  \begin{equation*}
    \alpha_k \sum_{i=1}^n \pi_i(k+1) \! \left\langle \! \frac{\mathbf{g}_{i,k}}{A_{ii}(k)}, \! \mathbf{y}_k-\mathbf{x}_{i,k} \! \right\rangle \! \le \! \alpha_k\! \chi\! L \!\sum_{i=1}^n \! \|\mathbf{y}_k-\mathbf{x}_{i,k}\|.
  \end{equation*}
  Since each $f_i$ is convex, we have
  \begin{equation}
    \langle \mathbf{g}_{i,k}, \mathbf{x} - \mathbf{x}_{i,k} \rangle \le f_i(\mathbf{x}) - f_i(\mathbf{y}_k) + f_i(\mathbf{y}_k) - f_i(\mathbf{x}_{i,k}).
  \end{equation}
  To derive $f(\mathbf{x})-f(\mathbf{y}_k)$, a little trick is applied:
  \begin{equation}
    \begin{aligned}
      \sum_{i=1}^n \frac{\pi_i(k+1)}{A_{ii}(k)} (f_i(\mathbf{x}) - f_i(\mathbf{y}_k)) \le \frac{1}{n}\left[f(\mathbf{x}) - f(\mathbf{y}_k)\right] + \\ \sum_{i=1}^n L \left| \frac{\pi_i(k+1)}{A_{ii}(k)} - \frac{1}{n} \right| \|\mathbf{y}_k-\mathbf{x}\|.
    \end{aligned}
  \end{equation}
  Considering the boundedness of $\mathcal{X}$, we have
  \begin{equation}
    \|\mathbf{y}_k - \mathbf{x}\| \le \eta + \sum_{i=1}^n \|\mathbf{y}_k - \mathbf{x}_{i,k}\|.
  \end{equation}
  Divide $\|\mathbf{y}_{k+1}-\mathbf{x}\|^2$ into three parts as $\|\mathbf{y}_k-\mathbf{x}\|^2$, $\|\mathbf{y}_{k+1}-\mathbf{y}_k\|^2$, and $\langle \mathbf{y}_{k+1}-\mathbf{y}_k, \mathbf{y}_k-\mathbf{x} \rangle$. We obtain the following iteration between $\|\mathbf{y}_{k+1}-\mathbf{x}\|^2$ and $\|\mathbf{y}_k-\mathbf{x}\|^2$:
  \begin{equation*}
    \begin{aligned}
      \mathbb{E}\|\mathbf{y}_{k+1}-\mathbf{x}\|^2 \le \mathbb{E}\|\mathbf{y}_k-\mathbf{x}\|^2 \! +\! \frac{2\alpha_k}{n}\mathbb{E}[f(\mathbf{x}) - f(\mathbf{y}_k)]\\
      + 4\zeta_k^2 + 4\zeta_k \sum_{i=1}^n \left( 2\mathbb{E}\|\mathbf{y}_k - \mathbf{x}_{i,k}\| + \mathbb{E}\|\mathbf{x}_{i,k+1} - \mathbf{x}_{i,k}\| \right)\\
      + 2 \zeta_k \! \left[ \eta + \sum_{i=1}^n \|\mathbf{y}_k - \mathbf{x}_{i,k}\| \right]  \left| \frac{\pi_i(k+1)}{A_{ii}(k)} - \frac{1}{n} \right|,
    \end{aligned}
  \end{equation*}
  where $\zeta_k = \alpha_k L \chi$. Since we have
  \begin{equation}
    \left| \frac{\pi_i(k+1)}{A_{ii}(k)} - \frac{1}{n} \right| \le \|I-A(k)\| + \left|\pi_i(k+1)-\frac{1}{n}\right|,
  \end{equation}
  it follows from Lemma \ref{lemma:conv-Gamma}, Lemma \ref{lemma:conv-xk-yk} and Theorem \ref{theorem:conv-absolute-prob-seq} that
  \begin{equation}
    \sum_{k=1}^\infty \alpha_k \mathbb{E}\left[ f(\mathbf{y}_k) - f(\mathbf{x}) \right] < \infty,
  \end{equation}
  which completes the proof.
\end{proof}

On the foundation of the lemmas above, the next theorem establishes the main convergence result for UDPSG.
\begin{theorem}
  Let Assumptions \ref{asp:graph}-\ref{asp:conv-time-varying-matrix} hold and $\{\mathbf{x}_{i,k}\}$ be generated by (\ref{eqn:unconstrained-iterates}). Then there exists a solution $\mathbf{x}^* \in \mathcal{X}^*$ to the optimization problem (\ref{eqn:optimization-objective}) such that for all $i$,
  \begin{equation}
    \lim_{k\rightarrow \infty} \mathbf{x}_{i,k} = \mathbf{x}^* \quad \text{a.s.}
  \end{equation}
\end{theorem}

\begin{proof}
  Since $\alpha_k$ is non-summable, (\ref{eqn:finiteness-of-f-dev}) indicates that 
  \begin{equation}
    \liminf_{k\rightarrow \infty} \mathbb{E}\left[ f(\mathbf{y}_k) - f(\mathbf{z}) \right] = 0,
  \end{equation}
  for any $\mathbf{z} \in \mathcal{X}^*$. Then there exists a subsequence $\{\mathbf{y}_{k_\ell}\}$ and a point $\mathbf{x}^*\in \mathcal{X}^*$ such that $\lim_{\ell \rightarrow \infty} \| \mathbf{y}_{k_\ell} - \mathbf{x}^*\| = 0$.
  
  Applying the result in \cite[Theorem 1]{1971ConvSupermartingale}, we conclude that $\|\mathbf{y}_k - \mathbf{z}\|^2$ converges for any $\mathbf{z} \in \mathcal{X}^*$ and hence for $\mathbf{x}^*$. Since the limit of a convergent sequence coincides with the limit of any convergent subsequence, we obtain $\lim_{k \rightarrow \infty} \| \mathbf{y}_{k} - \mathbf{x}^*\| = 0$, i.e., $\lim_{k \rightarrow \infty} \mathbf{y}_{k} = \mathbf{x}^*$. By Lemma \ref{lemma:conv-xk-yk}, we have $\lim_{k \rightarrow \infty} \mathbf{x}_{i,k} = \mathbf{x}^*$ a.s. $\forall i$.
\end{proof}

\vspace{-8pt}
\subsection{Non-Convex Local Objective Functions}
In practical multi-agent networks, the local functions are not necessarily convex. One of the most significant problems is to determine whether the decentralized algorithm converges when the local objective functions are smooth and non-convex, which are widely used in neural networks for deep learning applications. In this part, we will investigate the convergence behavior of UDPSG when the local objective functions are not convex but have several additional properties.

We will start with the Polyak-Lojasiewicz (PL) inequality \cite{1963POLYAK864}. Let $h$ be a smooth function defined in a real Hilbert space. $h$ is said to satisfy the PL inequality if for some $\mu > 0$, the following holds:
\begin{equation}
  \label{eqn:PL-cond}
  \frac{1}{2} \|\nabla h (x)\|^2 \ge \mu \left( h(x)- h^* \right), \quad \forall x,
\end{equation}
where $h^*$ is the global minimum of $h(x)$. Based on (\ref{eqn:PL-cond}), we conclude that UDPSG remains to be convergent if each local function satisfies the PL inequality:
\begin{lemma}
  \label{lemma:conv-PL-local-functions}
  Assume that each local function $f_i$ has an $L$-Lipschitz continuous gradient and satisfies PL inequality (\ref{eqn:PL-cond}). Letting Assumptions \ref{asp:graph}-\ref{asp:conv-time-varying-matrix} hold, we have
  \begin{equation}
    \lim_{k\rightarrow \infty} \mathbb{E}[f_i (\mathbf{x}_{i,k}) - f_i^*] = 0,
  \end{equation}
  where $f_i^*$ is the global minimum of $f_i$.
\end{lemma}

\begin{proof}
  We first analyze each term of $\langle \nabla f_i(\mathbf{x}_{i,k}), \mathbf{x}_{i,k+1}-\mathbf{x}_{i,k} \rangle$ by using the iteration in (\ref{eqn:xik-projection-error}) and (\ref{eqn:vik}):
  \begin{subequations}
    \begin{equation}
      \langle \nabla f_i(\mathbf{x}_{i,k}), \xi_{i,k} \rangle \le \|\xi_{i,k}\|,
    \end{equation}
    \begin{equation}
      \left\langle \nabla f_i(\mathbf{x}_{i,k}), \sum_{j=1}^n \mathbf{x}_{j,k}-\mathbf{x}_{i,k} \right\rangle \le 2L \sum_{i=1}^n \|\mathbf{x}_{i,k}-\mathbf{y}_k\|,
    \end{equation}
    \begin{equation}
      -\frac{\alpha_k}{A_{ii}(k)} \|\nabla f_i(\mathbf{x}_{i,k})\|^2 \le -\frac{2\mu \alpha_k}{A_{ii}(k)} (f_i(\mathbf{x}_{i,k})-f_i^*).
    \end{equation}
  \end{subequations}
  Since each $f_i$ has an $L$-Lipschitz continuous gradient, we have
  \begin{equation}
    \begin{aligned}
      f_i (\mathbf{x}_{i,k+1}) \le &f_i(\mathbf{x}_{i,k}) + \frac{L}{2} \|\mathbf{x}_{i,k+1} - \mathbf{x}_{i,k}\|^2 \\ &+ \langle \nabla f_i(\mathbf{x}_{i,k}), \mathbf{x}_{i,k+1}-\mathbf{x}_{i,k} \rangle.
    \end{aligned}
  \end{equation}
  From the derivation above, we conclude
  \begin{equation}
    f_i (\mathbf{x}_{i,k+1}) - f_i^* \le \left(1-\frac{2\mu \alpha_k}{A_{ii}(k)}\right) f_i(\mathbf{x}_{i,k}) - f_i^* + \lambda_{k},
  \end{equation}
  where $\lambda_k$ denotes the remainder. Unrolling the iteration and taking expectation at both sides, we have
  \begin{equation}
    \label{eqn:PL-recursive}
    \begin{aligned}
      \mathbb{E}[f_i (\mathbf{x}_{i,k+1}) - f_i^*] \le \prod_{t=1}^k \left(1-\frac{2\mu \alpha_t}{A_{ii}(t)} \right) \mathbb{E}[f_i^0 - f_i^*] \\ + \sum_{t=1}^k \prod_{r=1}^{k-t} \left(1-\frac{2\mu \alpha_r}{A_{ii}(r)} \right) \mathbb{E}\lambda_t,
    \end{aligned}
  \end{equation}
  where $f_i^0=f_i(\mathbf{x}_{i,0})$ is the initial value of $f_i$. Since $\lambda_k \rightarrow 0$ as $k \rightarrow \infty$, we obtain by Lemmas \ref{lemma:conv-product-non-summable-sequence}, \ref{lemma:conv-Gamma}, and \ref{lemma:conv-xk-yk}
  \begin{equation}
    \lim_{k\rightarrow \infty} \mathbb{E}[f_i (\mathbf{x}_{i,k}) - f_i^*] = 0, \quad \forall i,
  \end{equation}
  which completes the proof.
\end{proof}

Since each $\mathbf{x}_{i,k}$ converges to the global minima of $f_i$, we conclude that the global minimum of $f$ is obtained as follows:
\begin{equation*}
  \sum_{i=1}^n \lim_{k\rightarrow \infty} f_i(\mathbf{x}_{i,k}) \! = \! \lim_{k\rightarrow \infty} f(\mathbf{y}_k) \! \ge \! \min f(\mathbf{x}) \ge \sum_{i=1}^n \min f_i(\mathbf{x}). 
\end{equation*}

When each local objective function $f_i$ satisfies the PL inequality, it is direct to verify that the PL inequality holds for the global objective function $f=\sum f_i$ as well. The following proposition establishes the convergence of UDPSG when only the global objective function $f$ satisfies the PL inequality:
\begin{proposition}
  Assume that each local objective function $f_i$ and the global objective function $f$ have $L$-Lipschitz continuous gradients. Assume further that $f$  satisfies the PL inequality (\ref{eqn:PL-cond}). Letting Assumptions \ref{asp:graph}-\ref{asp:conv-time-varying-matrix} hold, we have
  \begin{equation}
    \lim_{k \rightarrow \infty} \mathbb{E} [f(\mathbf{x}_{i,k}) - f^*] = 0, \quad \forall i,
  \end{equation}
  where $f^*$ is the global minimum of $f$.
\end{proposition}

\begin{proof}
  Since the objective function $f$ has an $L$-Lipschitz continuous gradient, we have
  \begin{equation*}
    f(\mathbf{y}_{k+1}) \le f(\mathbf{y}_k) + \langle \nabla f(\mathbf{y}_k), \mathbf{y}_{k+1}-\mathbf{y}_k \rangle + \frac{L}{2} \|\mathbf{y}_{k+1}-\mathbf{y}_k\|^2.
  \end{equation*}
  Using $\pi(k) = \pi(k+1)A(k)$ and (\ref{eqn:y-iterates}), we obtain
  \begin{equation*}
    \begin{aligned}
      \mathbf{y}_{k+1}-\mathbf{y}_k \!=\! \sum_{i=1}^n \pi_i(k+1) \xi_{i,k} -\alpha_k \sum_{i=1}^n \frac{\pi_i(k+1)}{A_{ii}(k)}\nabla f_i (\mathbf{x}_{i,k}).
    \end{aligned}
  \end{equation*}
  We analyze each term respectively:
  \begin{subequations}
    \begin{equation}
      \mathbb{E}\left\langle \nabla f(\mathbf{y}_k), \xi_{i,k} \right\rangle \le L\zeta_k,
    \end{equation}
    \begin{equation}
      \begin{aligned}
        \sum_{i=1}^n \frac{\pi_i(k+1)}{A_{ii}(k)} \langle \nabla f(\mathbf{y}_k), \nabla f_i (\mathbf{x}_{i,k}) \rangle = \|\nabla f(\mathbf{y}_k)\|^2 \\
        + \sum_{i=1}^{n} \left( \frac{\pi_i(k+1)}{A_{ii}(k)} - 1 \right) \langle \nabla f(\mathbf{y}_k), \nabla f_i (\mathbf{x}_{i,k}) \rangle \\
        + \sum_{i=1}^n \left\langle \nabla f(\mathbf{y}_k), \nabla f_i(\mathbf{x}_{i,k})- \nabla f_i(\mathbf{y}_k) \right\rangle.
      \end{aligned}
    \end{equation}
  \end{subequations}
  According to the PL inequality in (\ref{eqn:PL-cond}), we have
  \begin{equation}
    -\alpha_k \|\nabla f(\mathbf{y}_k)\|^2 \le -2\mu \alpha_k (f(\mathbf{y}_k)-f^*),
  \end{equation}
  where $f^*$ is the global minimum of $f$. Then it follows that
  \begin{equation}
    \begin{aligned}
      \mathbb{E} \langle \nabla f(\mathbf{y}_k), \mathbf{y}_{k+1}-\mathbf{y}_k \rangle \le L\zeta_k -2\mu \alpha_k \mathbb{E}[f(\mathbf{y}_k)-f^*] \\
      +L^2 \chi (\sqrt{n}\|I-A(k)\|+ \Gamma(k,k)) \\ + L^2 \sum_{i=1}^n \|\mathbf{x}_{i,k} - \mathbf{y}_k\| .
    \end{aligned}
  \end{equation}
  Therefore, it is direct to obtain
  \begin{equation}
    \mathbb{E}[f(\mathbf{y}_{k+1}) - f^*] \le \prod_{t=1}^k \varrho_t \mathbb{E}[f^0 - f^*] + \sum_{t=1}^k \prod_{r=1}^{k-t} \varrho_r \mathbb{E} \lambda_t,
  \end{equation}
  where $\varrho_t = 1-2\mu\alpha_t$, and $\lambda_t$ stands for the other terms. Similar to (\ref{eqn:PL-recursive}), we conclude by the smoothness of $f$ that
  \begin{equation}
    \lim_{k \rightarrow \infty} \mathbb{E} [f(\mathbf{x}_{i,k}) - f^*] = 0, \quad \forall i,
  \end{equation}
  which completes the proof.
\end{proof}

In fact, functions that satisfy PL conditions belong to a wider category called invex functions. A function is invex if and only if every stationary point is a global minimum \cite{craven_glover_1985}. 

%% file: sections/simulation.tex
\vspace{-8pt}
\section{Simulation}
\label{sec:simulation}
In our simulation, the emphasis lies in evaluating the performance of UDPSG across three pivotal aspects: firstly, convergence for convex global objective functions; secondly, convergence for non-convex global objective functions, encompassing invex functions; and finally, a comparative analysis between UDPSG and standard DSG methods.
 
As for the generation of inhomogeneous stochastic matrices that satisfy Assumption \ref{asp:graph}, it is sufficient to produce a directed spanning tree at each iteration that is strongly connected, followed by the random inclusion of additional edges within the subgraph. Edge weights are uniformly sampled from the interval $(0,1]$ and subsequently normalized by row to ensure the resulting matrix is stochastic. We note that since the constant $\delta$ in Assumption \ref{asp:graph}c) can be set arbitrarily small (e.g., $\delta = 10^{-64n\log_2 n}$), the constraint can be neglected in practice. In terms of Assumption \ref{asp:conv-time-varying-matrix}, the constant $c$ can be set very close to $0$ if $\alpha_k=1/k$ such that the constraint on $A(k)$ is practically negligible as well. Therefore, the time-varying step size $\alpha_k = 1/k$ for convenience in the simulation. The projection space $\mathcal{X}$ is $[-1,1]^d$ where $d$ is determined by case.

We employ the proposed algorithm in addressing regression problems, a prevalent focus within decentralized machine learning applications, in a network composed of $n=6$ agents. Given a training set $\mathcal{D}=\{(a_i,b_i)\}$ for each agent, the objective of the network is to acquire a parameter $x$ that minimizes the global objective function $f=\sum_{i=1}^n f_i(x;a_i,b_i)$, where $f_i$ denotes the local objective function. 

\begin{figure}[t]
  \centering
  \includegraphics[width=0.48\textwidth]{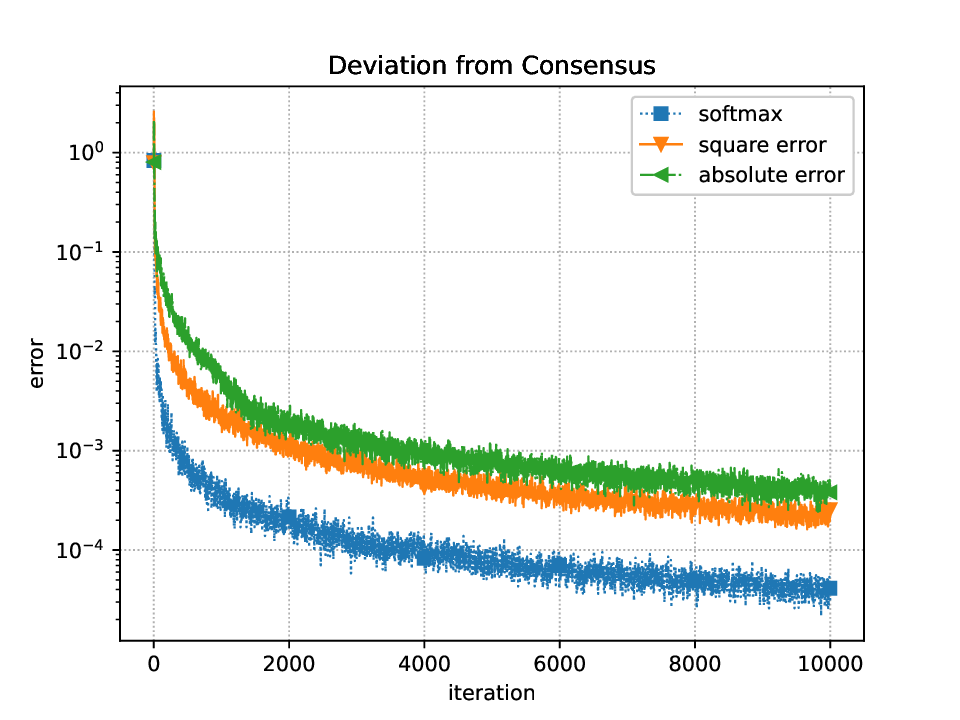}
  \caption{Average consensus errors in the case of local convex objective functions.}
  \label{fig:convex-consensus-error}
\end{figure}
\begin{figure}[t]
  \centering
  \includegraphics[width=0.48\textwidth]{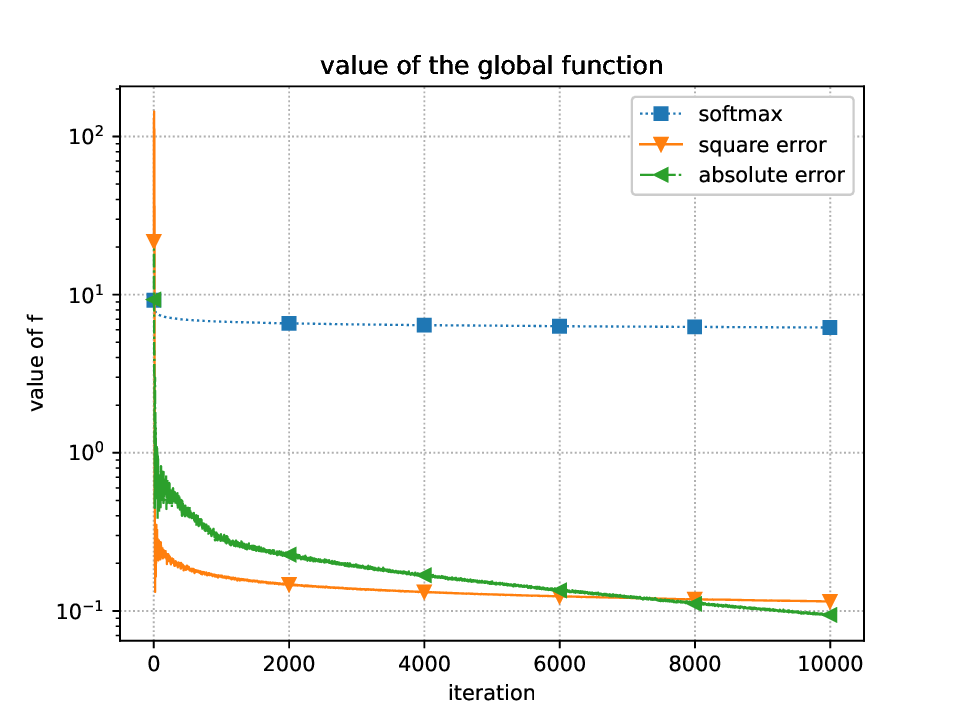}
  \caption{Value of the global convex objective functions.}
  \label{fig:convex-global-f-error}
\end{figure}

With regard to the convex local objective functions, we consider the following three kinds of functions widely used in machine learning: squared error ($f_i(x) = (\langle a_i,x \rangle - b_i)^2$), softmax ($f_i(x) = a_i \log(\sum_j e^{x_j})$), and absolute error ($f_i(x)=|\langle a_i,x \rangle-b_i|$). In addition, the computation of the average consensus error involves deriving the average norm by measuring the deviation between each agent's local parameter and the collective average parameter. As depicted in Fig. \ref{fig:convex-consensus-error}, the agents succeed in reaching consensus asymptotically under time-varying topology. Moreover, the minima of the global objective function is asymptotically attained in Fig. \ref{fig:convex-global-f-error}. 

\begin{figure}
  \centering
  \includegraphics[width=0.48\textwidth]{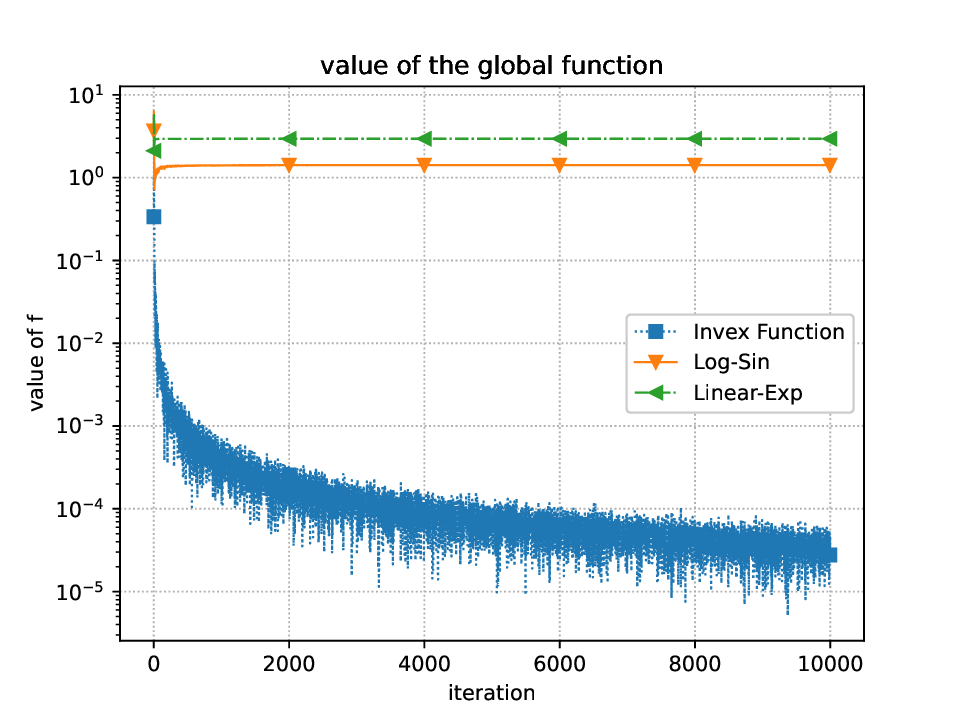}
  \caption{Value of the aggregated non-convex functions.}
  \label{fig:nonconvex-global-f-error}
\end{figure}

While the current study does not explicitly delve into the theoretical aspects pertaining to general non-convex objective functions, our interest remains focused on exploring the convergence behavior of the proposed algorithm in scenarios where the local objective functions exhibit non-convex characteristics. In the simulation, we test three kinds of non-convex functions as follows:
\begin{subequations}
  \begin{equation}
    \label{eqn:invex-func}
    \text{Invex Function: } f_i(x,y) = |a_i y| (b_i x^2 - 1)^2,
  \end{equation}
  \begin{equation}
    \text{Log-Sin: } f_i(x,y) = a_i \log(1+x^2+y^2) + \sin^2(x + b_i),
  \end{equation}
  \begin{equation}
    \text{Linear-Exp: } f_i(x,y) = (a_i y - 1/2)^2 e^{(x-b_i)^2}.
  \end{equation}
\end{subequations}
It is straightforward to check that (\ref{eqn:invex-func}) is invex and the other two functions are ``highly'' non-convex. As shown in Fig. \ref{fig:nonconvex-global-f-error}, the invex function asymptotically tends to its minimum $0$. Moreover, the proposed algorithm demonstrates the capability to identify, at the very least, a local minimum within a general non-convex function. Nevertheless, it remains challenging to ascertain whether a local minimum represents the global minimum due to the inherent complexity involved in analyzing the progression of a decentralized algorithm concerning general non-convex objective functions. 

\begin{figure}
  \centering
  \includegraphics[width=0.48\textwidth]{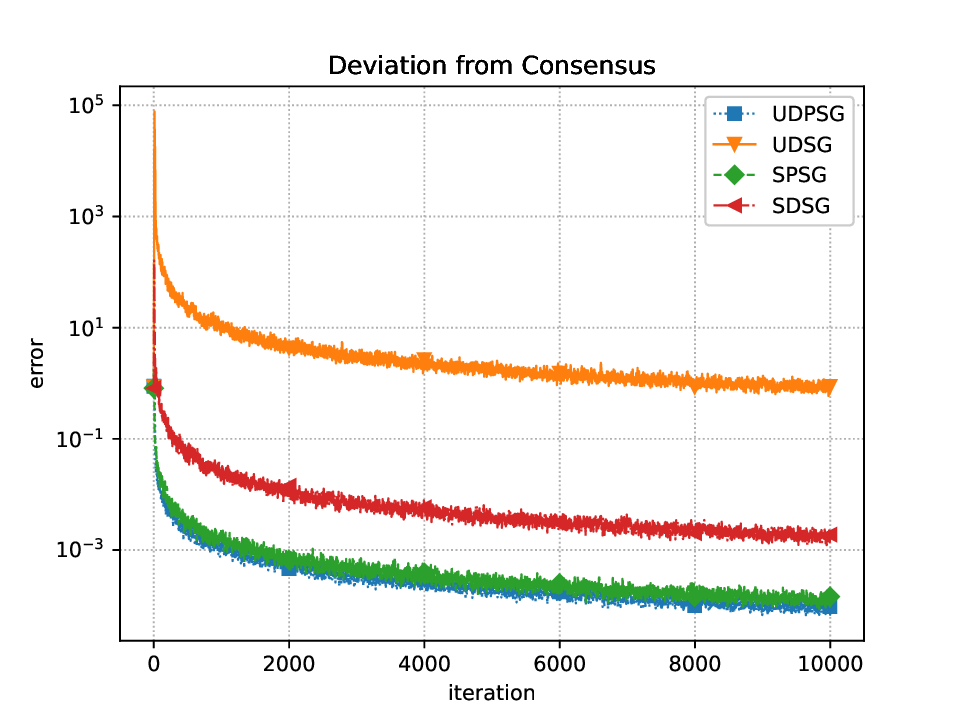}
  \caption{Average consensus errors with respect to different methods.}
  \label{fig:diff-methods-consensus-error}
\end{figure}
\begin{figure}
  \centering
  \includegraphics[width=0.48\textwidth]{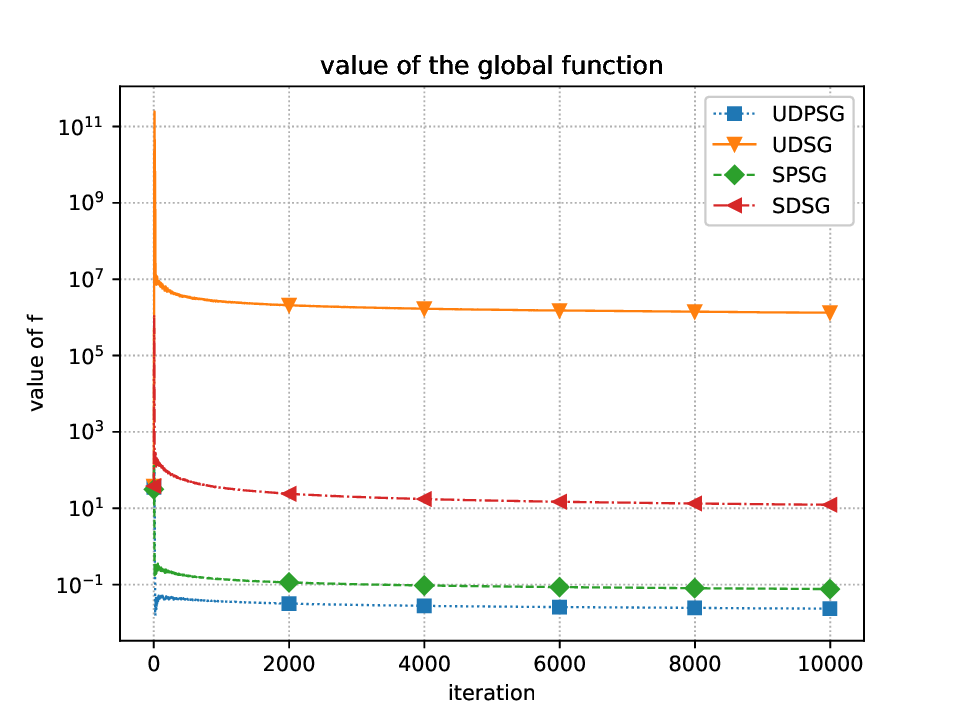}
  \caption{Value of the global objective function with respect to different methods.}
  \label{fig:diff-methods-global-f}
\end{figure}

Finally, we conduct a comparative experiment between UDPSG and standard DSG methods. We take the squared error function as the test case. Three baselines are used to better demonstrate the superior performance of UDPSG: UDSG (projection removed from UDPSG), standard distributed subgradient (SDSG), and standard distributed projected subgradient (SPSG). As deduced from Fig. \ref{fig:diff-methods-consensus-error} and Fig. \ref{fig:diff-methods-global-f}, all methods can reach consensus in terms of limiting behavior, which implies the universality of the convergence theory of the inhomogeneous infinite product. In addition, the topology-variant factor $A_{ii}(k)$ accelerates the convergence of UDPSG compared to a standard PSG method. The projection is significant in improving the performance of the subgradient method in speeding up consensus and identifying a global minima. 

%% file: sections/conclusion.tex
\vspace{-6pt}
\section{Conclusion}
\label{sec:conclusion}
This paper has developed novel theoretical results on the convergence properties of inhomogeneous IPSMs. The asymptotic behavior of the absolute probability sequence is analyzed when the stochastic matrices tend to the identity matrix. Based on these theoretical findings, we propose a novel decentralized projected subgradient algorithm for time-varying unbalanced directed graphs. This algorithm has been proven to converge to optimal solutions of distributed optimization problems when the objective functions are either convex or non-convex but satisfy PL conditions. Moreover, the performance of the proposed algorithm has been verified through numerical simulations. The theoretical results hold promise for broad applicability in analyzing a spectrum of decentralized algorithms, with the proposed algorithm serving as an illustrative instance. 

% It is our aspiration that these theoretical insights will catalyze the exploration of more comprehensive and widely applicable outcomes concerning time-varying graphs. We encourage further endeavors to explore and expand the applications of these findings, fostering deeper insights into decentralized algorithms and their practical implications.